\documentclass[orivec]{llncs}
\usepackage{textcomp}
\usepackage{amssymb}
\usepackage{verbatim}
\usepackage{array}
\usepackage{latexsym}
\usepackage{enumerate}
\usepackage{amsmath}
\usepackage{amsfonts}
\usepackage{color}
\usepackage[english]{babel}
\usepackage[notref,notcite]{showkeys}
\usepackage{comment}
\usepackage{units}
\usepackage{tikz-cd}

\newtheorem{result}[theorem]{Result}

\newtheorem{rem}[theorem]{Remark}
\usepackage{tikz-cd}

\def\cC{\mathcal C}
\def\cD{\mathcal D}
\def\cE{\mathcal E}
\def\cF{\mathcal F}
\def\cG{\mathcal G}

\def\cH{\mathcal H}

\def\cQ{\mathcal Q}

\def\cT{\mathcal T}

\def\cX{\mathcal X}
\def\cY{\mathcal Y}

\def\PG{{\rm{PG}}}
\def\ord{\mbox{\rm ord}}
\def\deg{\mbox{\rm deg}}

\def\div{\mbox{\rm div}}

\def\dim{\mbox{\rm dim}}

\def\Fs{{\mathbb F}_{q^2}}

\newcommand{\PGU}{\mbox{\rm PGU}}

\newcommand{\aut}{\mbox{\rm Aut}}

\newcommand{\g}{\gamma}

\newcommand{\go}{\omega}

\newcommand{\ha}{{\textstyle\frac{1}{2}}}

\definecolor{mypurple}{RGB}{145, 61, 58}
\definecolor{myblue}{RGB}{0, 128, 255}
\definecolor{mybro}{RGB}{158, 8, 8}
\definecolor{mygr}{RGB}{0, 135, 95}

\title{On a Galois cover of the Hermitian curve of genus $\mathfrak{g}=\frac{1}{8}(q-1)^2$}

\author{Barbara Gatti\inst{} \and Gioia Schulte\inst{}}
\institute{Department of Mathematics and Physics "Ennio de Giorgi", University of Salento, Lecce, Italy and\\ Department of Mathematics, Computer Science and Economics, University of the Basilicata,
Potenza, Italy\
\email{barbara.gatti@unisalento.it, giola.schulte@unisalento.it}}

\begin{document}


\maketitle
\begin{abstract}
In the study of algebraic curves with many points over a finite field, a well known general problem is to understanding better the properties of $\mathbb{F}_{q^2}$-maximal curves whose genera fall in the higher part of the spectrum of the genera of all $\mathbb{F}_{q^2}$-maximal curves. This problem is still open for genera smaller than $ \lfloor \frac{1}{6}(q^2-q+4) \rfloor$. 
In this paper we consider the case of $\mathfrak{g}=\frac{1}{8}(q-1)^2$ where $q\equiv 1\pmod{4}$ and the curve is the Galois cover of the Hermitian curve w.r.t to a cyclic automorphism group of order $4$. Our contributions concern Frobenius embedding, Weierstrass semigroups and automorphism groups.

\end{abstract}

\vspace{0.5cm}\noindent {\em Keywords}:
maximal curves, function fields, Galois cover, Frobenius embedding, Weierstrass semigroup.
\vspace{0.2cm}\noindent

\vspace{0.5cm}\noindent {\em Subject classifications}:
\vspace{0.2cm}\noindent  14H37, 14H05.

\begin{section}{Introduction} 
The interest around algebraic curves defined over a finite field $\mathbb{F}_\ell$ with many $\mathbb{F}_\ell$-rational points, especially $\mathbb{F}_\ell$-maximal curves, has been boasted by 
relevant applications to Finite geometry, Coding theory, and some other questions in Combinatorics. 
Here, $\mathbb{F}_\ell$-maximality means that the curve $\cF$ attains the upper bound in the famous Hasse-Weil theorem, that is, either $\cF$ is $\mathbb{F}_\ell$-rational, or $\ell=q^2$ for a power of the characteristic $p$ of $\mathbb{F}_q$ and the number of the points of a non-singular model $\cX$ defined over $\mathbb{F}_{q^2}$ is equal $q^2+1+2\mathfrak{g}q$ where $\mathfrak{g}>0$ is the genus of $\cF$. The problem of classifying $\mathbb{F}_{q^2}$-maximal curves has been completely solved so far for $q\le 7$, and partially for $q\le 16$; see \cite{nazar}. 

Among the $\mathbb{F}_{q^2}$-maximal curves for a given $q$, those with larger genera have clearly 
greater number of $\mathbb{F}_{q^2}$-rational points. This has given a motivation for the study of the upper spectrum of genera of $\mathbb{F}_{q^2}$-maximal curves. For a given $q$, the highest value of the spectrum, in other words the first largest genus, is $\ha q(q-1)$, and it is only attained by the Hermitian curve. 


The second largest genus is equal to $\lfloor \frac{1}{4}(q-1)^2\rfloor$, so it is quit far away from the highest value $\ha q(q-1)$. The second largest genus is realized by a unique curve. This is
\begin{itemize}
\item[(i)] $\cE^1_{\nicefrac{(q+1)}{2}}:Y^q+Y-X^{\nicefrac{(q+1)}{2}}=0$, $\mathfrak{g}=\frac{1}{4}(q-1)^2$, $q$ odd; see \cite{FT} and \cite[Example 10.3]{HKT},
\item[(ii)] $\mathcal{T}_2:X^{q+1}+Y+Y^2+Y^4+\cdots+Y^{\nicefrac{q}{2}}=0$, $\mathfrak{g}=\frac{1}{4}q(q-2)$, $q$ even; see \cite{AT} and \cite[Example 10.3]{HKT}.
\end{itemize}


The third largest genus is equal to $\lfloor \frac{1}{6}(q^2-q+4)\rfloor$, again there is a large distance from the second highest value in the spectrum. The curves which are known to have the third largest genus are three, namely
\begin{itemize}
\item[(iii)] $\cF_0:X^{\nicefrac{(q+1)}{3}}+X^{\nicefrac{2(q+1)}{3}}+Y^{q+1}=0$, $\mathfrak{g}=\frac{1}{6}(q^2-q+4)$, $q\equiv 2 \pmod{  3}$; see \cite{CKT2}, and \cite[Example 10.33]{HKT},
\item[(iv)] $\cG:Y^q-YX^{\nicefrac{2(q-1)}{3}}+X^{\nicefrac{(q-1)}{3}}=0$, $\mathfrak{g}=\frac{1}{6}(q^2-q)$, $q\equiv 2 \pmod{3}$; see \cite{GSX} and  \cite[Example 10.34]{HKT}, 
\item[(v)] $\cT'_3:(Y+Y^3+\cdots+Y^{\nicefrac{q}{3}})^2-X^q-X=0$, $\mathfrak{g}=\frac{1}{6}(q^2-q)$, $q=3^h$; see \cite{AT2} and \cite[Example 10.37]{HKT}.
\end{itemize}
There are known three more $\mathbb{F}_{q^2}$-maximal curves whose genera are close to $\lfloor \frac{1}{6}(q^2-q+4)\rfloor$. They are 
\begin{itemize}
\item[(vi)] $\cF'_0: YX^{\nicefrac{(q-2)}{3}}+Y^q+X^{\nicefrac{(2q-1)}{3}}=0$, $\mathfrak{g}=\frac{1}{6}(q^2-q-2)$, $q\equiv 2 \pmod{3}$; see \cite{CKT2} and \cite[Example 10.34]{HKT}, 
\item[(vii)] $\mathcal{E}^1_{\nicefrac{(q+1)}{3}}: Y^q+Y-X^{\nicefrac{(q+1)}{3}}=0$, $\mathfrak{g}=\frac{1}{6}(q-2)(q-1)$, $q\equiv 2 \pmod{3}, q\geq 11;$ see \cite{KTA} and  \cite[Section 10.6]{HKT},
\item[(viii)] $\cT^{''}_{3}: Y+Y^3+\cdots+Y^{\nicefrac{q}{3}}+\go X^{q+1}=0$, $\go^{q-1}=-1$, $\mathfrak{g}=\frac{1}{6}q(q-3)$, $q=3^h$; see \cite{AT2} and \cite{CKT1}.
\end{itemize}
To find the next value in the spectrum which is the genus of some known $\mathbb{F}_{q^2}$-maximal curve for larger $q$, one has to get to $\lfloor \frac{1}{8}(q^2-2q+5)\rfloor$. We know $\mathbb{F}_{q^2}$-maximal curves for that genus and explicit equations are available for two of them. They  are  
\begin{itemize}
\item[(ix)] $\alpha_0(X)+\alpha_1(X)Y+\ldots+\alpha_i(X)Y^{2^i}+\ldots+ \alpha_{h-1}(X)Y^{2^{h-1}}=0$
where $\alpha_i(X)\in \mathbb{F}_{q^2}[X]$ and
$$\alpha_0(X)=X^{q+1},\,\, \alpha_1(X)= \frac{(X^q+X)^2+(X+b+b^2)^q(X^q+X)}{T(X)},\,\,\alpha_{h-1}(X)=(X+b+b^2)^{2q}$$
and the other coefficients $\alpha_i(X)$ for $2\le i \le h-2$ are computed recursively from the equation $T(X)\alpha_i(X)+\alpha_{i-1}(X)^2+(X+b+b^2)^{2q}+(X+b+b^2)^q(X^q+X)$
where $T(X)=X+X^p+\ldots+X^{\nicefrac{q}{p}}$,
$\mathfrak{g}=\frac{1}{8}q(q-2)$, $q\equiv 0 \pmod{4}$; see \cite{GalSub},
\item[(x)] $X^{\nicefrac{(q^2-1)}{4}}=Y(Y+1)^{q-1}$, $\mathfrak{g}= \frac{1}{8}(q-1)^2$, $q\equiv 1 \pmod{4}$; see \cite[Example 6.3]{GSX}.
\end{itemize}
The existence of a third $\mathbb{F}_{q^2}$-maximal curve of genus $\lfloor \frac{1}{8}(q^2-2q+5)\rfloor$ for $q\equiv 3 \pmod{4}$ is also known, but an explicit equation is still missing. 
There are known four more $\mathbb{F}_{q^2}$-maximal curves whose genera are close to $\lfloor \frac{1}{8}(q^2-2q+5)\rfloor$. They are  
\begin{itemize}
\item[(xi)] $\mathcal{E}^1_{\nicefrac{(q+1)}{4}}: Y^q+Y-X^{\nicefrac{(q+1)}{4}}=0$, $\mathfrak{g}=\frac{1}{8}(q-1)(q-3)$, $q\equiv 1 \pmod{4}$; see \cite{KTA} and  \cite[Section 10.6]{HKT},
\item[(xii)] $\mathcal{E}^1_{\nicefrac{(q+1)}{4}}:Y^q+Y-X^{\nicefrac{(q+1)}{4}}=0$,
$\cD_{\nicefrac{(q+1)}{2}}:X^{\nicefrac{(q+1)}{2}}+Y^{\nicefrac{(q+1)}{2}}+1=0$, $\mathfrak{g}=\frac{1}{8}(q-1)(q-3)$, $q\equiv 3 \pmod{4}$; see \cite{KTA} and  \cite[Section 10.6]{HKT},
\item[(xiii)] $\cX_2: Y^4+aY^2+Y+X^{q+1}=0$, $\mathfrak{g}=\frac{1}{8}q(q-4)$, $q\equiv 0 \pmod{2}$; see \cite{AG} and  \cite{HKT}.
\end{itemize} 

For a better outstanding of the relevant features of $\mathbb{F}_{q^2}$-maximal curves of genera in the upper part of the spectrum, it seems indispensable to investigate the main characteristics of the known $\mathbb{F}_{q^2}$-maximal curves of higher genera, first of all, their Frobenius embedding, Weierstrass semigroups at some $\mathbb{F}_{q^2}$-rational points, and automorphism group defined over $\mathbb{F}_{q^2}$. The results obtained may give a hint to figure out the largest value $\mathfrak{g}_0$ in the spectrum which is the genus of some $\mathbb{F}_{q^2}$-maximal curve not covered by the Hermitian curve. From previous work, $\mathfrak{g}_0\ge \ha (n^3+1)(n^2-2)+1$ where $q=n^3$; see \cite{GK}. In particular, the known $\mathbb{F}_{q^2}$-maximal curves of genera in the upper part of the spectrum are Galois sub-covers of the Hermitian curve.       

Our contributions in this direction are stated in the following theorem.
\begin{theorem}
\label{mt16102024} Let $\cF$ the $\mathbb{F}_{q^2}$-maximal curve as in {\rm{(x)}}. Then
\begin{itemize}
\item[(I)] The Frobenius embedding $\cX$ of $\cF$ has dimension four, and for $p>3$ the order sequence of $\cX$ is $(0,1,2,3,q)$.  
\item[(II)] There is a $\mathbb{F}_{q^2}$-rational point $P$ of $\cX$ such that the non-gaps of $\cX$ at $P$ smaller than $q+1$ are $\ha(q+1), \frac{1}{4}(q-1),q$.
\item[(III)] The $\mathbb{F}_{q^2}$-rational automorphism group of $\cX$ has order $\ha (q^2-1)$ and it is the inherited group from the Hermitian curve. 
\end{itemize}
\end{theorem}
From our investigation some new features have emerged. In particular, the Frobenius embedding $\cX$ of $\cF$ is related with the classical theorem of Halphen; see \cite[Lemma 7.119]{HKT}. In fact, $\mathfrak{g}(\cX)$ exceeds the Halphen number $c_1(q+1,4)$ just by $1$, and it shows that the Halphen theorem is sharp, as $\cX$ is not contained in surface of $PG(4,\mathbb{F})$ of degree $\le 3$.

Our notation and terminology are standard and come from \cite{HKT}, see also \cite{Se} and \cite{sti}. Surveys on maximal curves are found in \cite{ACP,GA,GA1,GA2,VG} and  \cite[Section 10]{HKT}.

Throughout the paper, $\cH_q$ denotes the Hermitian curve given by its canonical affine equation $Y^q+Y-X^{q+1}=0$, and $\mathbb{F}_{q^2}(\cH_q)=\mathbb{F}_{q^2}(x,y)$ with $y^q+y-x^{q+1}=0$ stands for the $\mathbb{F}_{q^2}$-rational function field of $\cH_q$.

\end{section}

\section{Background on function fields}
Let $\cC$ be an $\mathbb{F}_\ell$-\emph{rational plane curve}, that is, a projective, geometrically irreducible, possible singular algebraic plane curve of genus $\mathfrak{g}$
defined over the finite field $\mathbb{F}_\ell$ of order $\ell$. Let $\mathbb{F}_\ell(\cC)$ be the \emph{function field} of $\cC$ which is an algebraic function field of transcendency degree one with constant field $\mathbb{F}_\ell$. As usual, $\cC$ is  viewed as a curve over an algebraic closure $\mathbb{F}$ of $\mathbb{F}_\ell$, and its function field $\mathbb{F}(\cC)$ is a constant field extension of $\mathbb{F}_\ell(\cC)$ by $\mathbb{F}|\mathbb{F}_\ell$. 

For any place $P$ of $\mathbb{F}(\cC)$, the \emph{Weierstrass semigroup} $H(P)$ of $\cC$ is a subsemi-group of the additive semigroup of $\mathbb{N}$ whose elements are the natural numbers $n$ such that there exists a function of $\mathbb{F}(\cC)$ whose pole divisor is $nP$. The numbers in the complementary set $G(P)=\mathbb{N}\setminus H(P)$ are the \emph{gaps}. The Weierstrass gap Theorem, see \cite[Theorem 6.89]{HKT}, states that there are exactly $\mathfrak{g}$ gaps at $P$, and that they fall in the interval $[1,2\mathfrak{g}-1]$. 
The distribution of the gaps in $[1,2\mathfrak{g}-1]$, equivalently the structure of $H(P)$, does not alter when $P$ varies on the set of all places of $\mathbb{F}(\cC)$, apart from finitely many exceptional places, named     Weierstrass points of $\cC$. 

There exists an one-to-one correspondence between places of $\mathbb{F}(\cC)$ and branches of $\cC$. According to \cite[Definition 4.29]{HKT} see also \cite{Se}, a \emph{branch} is an equivalence class of primitive branch representations where a branch representation is a point of the projective plane over the field $\mathbb{F}((t))$ of all formal power series in the indeterminate $t$ and with coefficients in $\mathbb{F}$. Center, order and order sequence of a branch are defined to be the center, the order, and the order sequence of any of its branch representations. Every point $P$ of $\cC$ is the center of at least one branch, and the total number of branches centered at $P$ does not exceed the multiplicity of $P$, and this bound is attained if and only if $P$ is a node. Fix a homogeneous coordinate system $(X_0:X_1:X_2)$ in the projective plane $PG(2,\mathbb{F})$ coordinatized by $\mathbb{F}$. Regarding the line $\ell_\infty$ of equation $X_0=0$, the arising affine plane $AG(2,\mathbb{F})$ coordinatized by $\mathbb{F}$ has an affine coordinate system $(X,Y)$ where $X=X_1/X_0$ and $Y=X_2/X_0$. Let $f(X,Y)=0$ be a (minimal, affine) equation of $\cC$. Then $\mathbb{F}(\cC)=\mathbb{F}(x,y)$ with $f(x,y)=0$. If $P=(a_0:a_1:a_2)$ is a point of $\cC$ with $a_0\ne 0$, i.e. $P$ is a point of $AG(2,\mathbb{F})$, then each primitive branch representation $\gamma$ of $\cC$ centered at $P$ has a representation $(x(t),y(t))$ in \emph{special affine coordinates} where $a=a_1/a_0$, $b=a_2/a_0$, and $x(t)=a+\alpha(t),\,y(t)=b+\beta(t)$ with $\alpha(t),\beta(t)\in \mathbb{F}[[t]]$. In particular, $P$ is the center of the branch represented by $\gamma$ and the branch belongs to the curve $\cC$ of equation $f(X,Y)=0$ if and only if $f(x(t),y(t))$ vanishes in $\mathbb{F}[[t]]$. If $a_0=0$, i.e. $P$ is a point of $PG(2,\mathbb{F})$ not in $AG(2,\mathbb{F})$, then at last one of the other two coordinates $a_1,a_2$ is not zero, say $a_1$. Changing the homogeneous coordinate system $(X_0:X_1:X_2)$ into $(X_1:X_0:X_2)$ ensures that $a_0\neq 0$, and hence the above discussion on branches of $\cC$ centered at $P$ can be repeated.       

The \emph{automorphism group} $\aut(\cC)$ of $\cC$ is defined to be the automorphism group of $\mathbb{F}(\cC)$ fixing every element of $\mathbb{F}$. It has a faithful permutation representation on the set of all places of $\mathbb{F}(\cC)$ (equivalently on the set of all branches of $\mathbb{F}(\cC))$. The \emph{$\mathbb{F}_\ell$-automorphism group}  $\aut({\mathbb{F}_\ell}(\cC))$ of $\mathbb{F}_\ell(\cC)$ is the subgroup of $\aut(\cC)$ consisting of all automorphisms of $\aut(\cC))$ defined over $\mathbb{F}_\ell$. In particular, the action of $\aut({\mathbb{F}_\ell}(\cC))$ on the $\mathbb{F}_\ell$-rational branches of $\cC$ is the same as on the set of degree $1$ places of $\mathbb{F}_\ell(\cC)$.

Let $G$ be a finite subgroup of $\aut({\mathbb{F}_{\ell}}(\cC))$. The \emph{Galois subcover} of $\mathbb{F}_\ell(\cC)$ with respect to $G$ is the fixed field of $G$, that is, the subfield ${\mathbb{F}_{\ell}}(\cC)^G$ consisting of all functions in $\mathbb{F}_{\ell}(\cC)$ fixed by every element in $G$. 
There exists a projective, geometrically irreducible, possible singular algebraic plane curve $\cY$ whose function field $\mathbb{F}_{\ell}(\cY)$ is isomorphic to ${\mathbb{F}_{\ell}}(\cC)^G$ over $\mathbb{F}_\ell$. 
Such a curve $\cY$ is unique, apart from birational isomorphisms over $\mathbb{F}_\ell$, and it is the \emph{quotient curve of $\cC$ with respect to $G$} and is denoted by $\cC/G$. The covering $\cC\mapsto \cY$ has degree $|G|$ and the field extension $\mathbb{F}_{\ell}(\cC)|\mathbb{F}_{\ell}(\cC)^G$ is Galois.

Now, we collect those known results on automorphisms which will be used in our proofs. Let $G$ be any subgroup of $\aut(\cC)$. 
If $P$ is a place of $\mathbb{F}(\cC)$, then the \emph{stabiliser} $G_P$ of $P$ in $G$ is the subgroup of $G$ consisting of all elements fixing $P$. 

From now on let $\ell=q^2$ with $q=p^h$ and assume that $\cX$ is a $\mathbb{F}_{q^2}$-maximal curve.
\begin{result}\cite[Theorem 11.49(b)]{HKT}
\label{resth11.49b} All $p$-elements of $G_P$ together with the identity form a normal subgroup $S_P$ of $G_P$ so that $G_P=S_P\rtimes C$, the semidirect product of $S_P$ with a cyclic complement $C$.
\end{result}
\begin{result}\cite[Theorem 11.129]{HKT}
\label{resth11.129} If $\cX$ has zero Hasse-Witt invariant then every non-trivial element of a $p$-subgroup $S_p$ of   $\aut(\cX)$ has a unique fixed point $P$ on $\cX$, and no non-trivial element in $S_p$ fixes a point other than $P$.
\end{result}
The following result is well known, see for instance \cite{GT}. 

\begin{result}
\label{zeroprank} All $\mathbb{F}_{q^2}$-maximal curves have zero Hasse-Witt invariant.
\end{result}
The following result is commonly attributed to Serre, see Lachaud \cite{lachaud1987}.
\begin{result}\cite[Theorem 10.2]{HKT}
\label{resth10.2} For every subgroup $G$ of $\aut({\mathbb{F}_{q^2}}(\cX))$, the quotient curve $\cX/G$ is also  $\mathbb{F}_{q^2}$-maximal.
\end{result}

\subsection{The Hermitian curve and its function field}
 In this subsection, we focus on the Hermitian function fields and its Galois subcovers. The usual affine equation, or canonical form, of the Hermitian curve $\cH_q$, as an $\mathbb{F}_{q^2}$-rational curve, is $Y^q+Y=X^{q+1}$ and hence its function field is $\mathbb{F}_{q^2}(x,y)$ with $y^q+y-x^{q+1}=0$. 
Since the Hermitian curve is non-singular, each point of $\cH_q$ is the center of a unique branch. So we may identify each point of $\cH_q$ with the unique branch centered at the point.     
 We collect a number of known results on the $\mathbb{F}_{q^2}$-automorphism group $\aut(\mathbb{F}_{q^2}(\cH_q))$ of $\cH_q$. For more details, the reader is referred to  \cite{hoffer1972,huppertI1967}.
\begin{result}\cite[Theorem 12.24 (iv), Proposition 11.30]{HKT}
\label{sect12.3}
$\aut(\mathbb{F}_{q^2}(\cH_q))\cong \PGU(3,q)$ and $\aut(\mathbb{F}_{q^2}(\cH_q))$ acts on the set of all $\mathbb{F}_{q^2}$-rational points of $\cH_q$ as $\PGU(3,q)$ in its natural doubly transitive permutation  representation of degree $q^3+1$ on the isotropic points of the unitary polarity of the projective plane $\PG(2,\mathbb{F}_{q^2})$. Furthermore,
$\aut(\mathbb{F}_{q^2}(\cH_q))$ $\cong$ $\aut(\mathbb{F}(\cH_q))$.
\end{result}
The following result based on Result \ref{sect12.3} describes the structure of the two-point stabiliser $H$ of $\aut(\mathbb{F}_{q^2}(\cH_q))$. Since $\aut(\mathbb{F}_{q^2}(\cH_q))$ acts on the set of $\mathbb{F}_{q^2}$-rational points of $\cH_q$ as doubly transitive permutation group, the origin $O=(0,0)$ together with the unique point at infinity $Y_\infty$ of $\cH_q$ may be chosen so that $H$ is the stabiliser of $O$ and $Y_\infty$. 

\begin{result}\cite[Section 4]{GSX}
\label{struct} Let the Hermitian function field be given by its canonical form $\mathbb{F}_{q^2}(x,y)$ with $y^q+y-x^{q+1}=0$. Then the stabiliser $H$ of both points $O$ and $Y_\infty$ in $\aut(\mathbb{F}_{q^2}(\cH_q))$ is a cyclic group of order $q^2-1$ and it consists of all maps
\begin{equation*}
    \varphi_{\lambda}:\,(x,y)\mapsto (\lambda x,\lambda^{q+1}y)
 \end{equation*}
 where $\lambda$ ranges over the non-zero elements in $\mathbb{F}_{q^2}$.
 \end{result}
 Let $\Phi$ be the subgroup of $H$ of order $4$. The following result is due to Garc\'ia, Stichtenoth and Xing. 
 \begin{result}\cite[Example 6.3]{GSX} 
 \label{gsxA} The $\mathbb{F}_{q^2}$-rational plane curve of equation 
 \begin{equation*}
 z^{\nicefrac{(q^2-1)}{4}}=t(t+1)^{q-1} 
 \end{equation*}
 is the quotient curve of $\cH_q$ with respect to $\Phi$, and, for $q\equiv 1\pmod 4$, it has genus $\mathfrak{g}=\frac{1}{8}(q-1)^2$.  
 \end{result}


\subsection{Background on algebraic curves}
We report some background from \cite[Chapter 7]{HKT}, see also \cite{sv}, about the geometry of curves embedded in projective spaces.
\subsubsection{Order sequence} Let $\cX$ be a projective, geometrically irreducible, non-singular algebraic curve, embedded in a projective space ${\rm{PG}}(r,\mathbb{F})$ where $\mathbb{F}$ is an algebracially closed field of characteristic $p>0$.
For any point $P\in\cX$ 
the possible intersection multiplicities of hyperplanes with the curve at $P$ are considered. There is only a finite number of these intersection multiplicities, the number being equal to $r+1$. There is a unique hyperplane, called osculating hyperplane with the maximum intersection multiplicity. The hyperplanes cut out on $\Gamma$ a simple,
fixed point free, not-necessarily complete linear series $\Sigma$ of dimension $r$ and degree $n$, where $n$ is the degree of the curve $\cX$. An integer $j$ is a $(\Sigma,P)$-order if there is a hyperplane $H$ such that $I(P,H\cap \cX)=j$. 
In the case that $\Sigma$ is the canonical series, it follows from the
Riemann--Roch theorem that $j$ is a $(\Sigma,P)$-order if and only if $j+1$ is a Weierstrass gap.

For any non-negative integer $i$, consider the set of all hyperplanes $H$
of $\PG(r,\mathbb{F})$ for which the intersection number is at least $i$. Such
hyperplanes correspond to the points of a subspace $\overline{\Pi}_i$ in
the dual space of $\PG(r,K)$. Then we have  the decreasing chain
$$
\PG(r,K) = \overline{\Pi}_0 \supset
\overline{\Pi}_1 \supset \overline{\Pi}_2 \supset \cdots.
$$

An integer $j$ is a $(\Sigma,P)$-order if and only if $\overline{\Pi}_j$ is not equal to the subsequent space in the chain.
In this case $\overline{\Pi}_{j+1}$ has
codimension 1 in $\overline{\Pi}_j$. Since $\deg\,\, \Sigma=n$, we have that $\overline{\Pi}_i$ is
empty as soon as $i> n$. The number of $(\Sigma,P)$-orders is exactly
$r+1$; they are $j_0(P),j_1(P),\ldots,j_r(P)$ in increasing order, and
$(j_0(P),j_1(P),\ldots,j_r(P))$ is the order-sequence of $\cX$ at $P$.
Here $j_0(P)=0$, and
$j_1(P)=1$ since $P$ is a non-singular point of $\cX$.
Consider the intersection $\Pi_i$ of hyperplanes $H$ of $\PG(r,K)$,
for which
$$
I(P,H\cap\g)\geq j_{i + 1}.
$$ Then the flag
$
\Pi_0\subset \Pi_1\subset \cdots \subset \Pi_{r-1} \subset
\PG(r,K)
$
can be viewed as the algebraic analogue of the Frenet
frame in differential geometry.
 Notice that $\Pi_0$ is just $P$, and $\Pi_1$ is the tangent line to the
branch $\g$ at $P$. 
Furthermore, $\Pi_{r-1}$ is the osculating hyperplane at $P$.
The order-sequence is the same for all but finitely many points of $\cX$, each such exceptional point is called a $\Sigma$-Weierstrass point of $\cX$.
The order-sequence at a generally chosen point of $\cX$ is the order sequence of $\cX$ and denoted by $(\varepsilon_0,\varepsilon_1,\ldots,\varepsilon_r)$. Here $j_i(P)$ is at least $\varepsilon_i$ for $0\le i \le r$ at any point of $\cX$.
\subsubsection{Projections}
We also recall how the \emph{projection} $\pi$ of $\cX$ from a point $C\in PG(r,\mathbb{F})$ on a hyperplane $\Pi$ disjoint from $C$ is performed; see \cite[Chapter 7]{HKT}.  For a point $P\in\cX$ other than $C$, the line through $C$ and $P$ meets $\Pi$ in exactly one point, the projection $P'=\pi(P)$ of $P$. If $C\in \cX$ then $C$ also has a projection which is defined to be the point cut out on $\Pi$ by the tangent line to $\cX$ at $C$. There exists a geometrically irreducible, not necessarily non-singular, algebraic curve $\cY$ in $PG(r-1,\mathbb{F})$
which contains all but a finitely many the projections of the points of $\cX$. Actually, $\pi$ may happen not to be  birational. If this is the case then $\cY$ is \emph{covered} by $\cX$, and $\pi$ is \emph{composed of an involution} of degree $d$ where the geometric meaning of $d$ is the number of points of $\cX$ with the same  projection on a generically chosen point of $\cY$. In terms of function fields, $\mathbb{F}(\cY)$ is isomorphic to a subfield of $\mathbb{F}(\cX)$ of index $d$. If $\mathbb{F}(\cY)$ is the fixed field of an automorphism group $G$ of $\mathbb{F}(\cX)$ then $\cY$ is the quotient curve of $\cX$ with respect to $G$. 
\subsubsection{B\'ezout's theorem}
The higher dimensional generalization of B\'ezout's theorem about the number of common points of $m-1$ hypersurfaces $\cH_1,\ldots,\cH_{m-1}$ of $\textrm{PG}(m-1,\mathbb{F}_q)$ states that either that number is infinite or does
not exceed the product $\deg(\cH_1)\cdot\ldots \cdot\deg(\cH_{m-1})$.
For a discussion on B\'ezout's theorem and its generalization; see \cite{Vog}.

\subsubsection{Embedding of a plane curve in a higher dimension projective space} Let $\cC$ be  a projective, geometrically irreducible, possible singular algebraic plane curve defined over $\mathbb{F}$. Take $r+1$ variables from the function field $\mathbb{F}(\cC)$, say $u_0,u_1,\ldots u_r$, which are linearly independent over $\mathbb{F}$. The \emph{embedding} of $\cC$ via $(u_0,u_1,\ldots,u_r)$ is the algebraic curve $\cX$ of $PG(r,\mathbb{F})$ whose branches are the images of those of $\cC$ under the map $\tau:\,(x,y)\mapsto (u_0,u_1,\ldots,u_r)$. From now on we consider the case where $\tau$ is birational and $\cX$ is geometrically irreducible and non-singular, that is, $\cX$ is a non-singular model of $\cC$. For the purpose of the present paper, we also assume that $r=4$. Then, apart from the case $p=3$ we have $p>r$, and hence we may use ordinary higher derivatives instead of Hasse derivatives. Another admissible simplification is to put $u_0=1,u_1=x,u_2=y$ where $\cC$ has equation $f(X,Y)=0$. Set $z=u_3, u=u_4$. Then $z=z(x,y)$ and $u=u(x,y)$ with $z(x,y),u(x,y)\in \mathbb{F}(\cC)$. For a branch $\gamma$ of $\cC$, let $(x(t),y(t))$ with $x(t),y(t)\in \mathbb{F}[[t]]$ be a primitive branch representation of $\gamma$. Then the image $P\in\cX$ of $\gamma$ has a primitive branch representation $(1,x(t),y(t),z(t),u(t))$ where $z(t)=z(x(t),y(t))$, and $u(t)=u(x(t),y(t))$. For the derivatives and higher derivatives of $w(t)\in \mathbb{F}[[t]]$ we write $w^{(i)}$ in place of $d^iw(t)/dt^i$. 
Then, the Wronskian matrix $W$ of $\cX$ at the point $P$ reads   
\begin{equation*}
W=
\begin{pmatrix}
 1& x(P) & y(P) & z(P) & u(P)\\
 0& 1 & y^{(1)}(P) & z^{(1)}(P) & u^{(1)}(P)\\
 0 & 0 & y^{(2)}(P) & z^{(2)}(P) & u^{(2)}(P)\\
 0 & 0 & y^{(3)}(P) & z^{(3)}(P) & u^{(3)}(P)\\
 0 & 0 & y^{(4)}(P) & z^{(4)}(P) & u^{(4)}(P)
\end{pmatrix}.
\end{equation*}
The first four orders $j_0(P),j_1(P),j_2(P),j_3(P)$ are $0,1,2,3$ if and only if the first four rows in $W$ are linearly independent over $\mathbb{F}$. Moreover, $W$ has non-vanishing determinant if and only if $j_4(P)=4$.  

\subsubsection{Natural Embedding} In this subsection, $\cC$ is an $\mathbb{F}_{q^2}$-rational plane curve. 
For a degree one place $P_0$ of $\mathbb{F}_{q^2}(\cC)$,  $|(q+1)P_0|$ is the \emph{Frobenius} linear series.
Assume that $\cC$ is an $\mathbb{F}_{q^2}$-maximal curve. Then the Frobenius linear series does not depend on the choice of $P_0$, i.e. $(q+1)P_0$ and $(q+1)Q$ are equivalent divisors for any degree one $\mathbb{F}_{q^2}$-rational place $Q$ of $\mathbb{F}_{q^2}(\cC)$. Moreover, the linear series $|(q+1)P_0|$ is complete, base-point free, simple and defined over $\mathbb{F}_{q^2}$. 
For $r=\dim(|(q+1)P_0|)$, let $u_0,u_1,\ldots,u_r$ be a basis for the Riemann-Roch space $\mathcal{L}((q+1)P_0)$. Then the embedding $\tau$ of $\cC$ via $(u_0,u_1,\ldots,u_r)$ into $PG(r,\mathbb{F})$ is the \emph{Frobenius}-embedding, and $r=\dim(|(q+1)P_0|)$ is the \emph{Frobenius}-dimension. The natural embedding theorem states that $\cX=\tau(\cC)$ is a non-singular model of $\cC$ of degree $q+1$ which lies in a Hermitian variety $\cH_{r,q}$ of $PG(r,\mathbb{F}_{q^2})$; see \cite[Section 10.2]{HKT}. It may happen that $\cH_{r,q}$ is singular. If this is the case, a projection from a vertex in  $PG(r,\mathbb{F}_{q^2})$ maps  $\cX$ to another non-singular model $\cX'$ of $\cC$ embedded in $PG(r',\mathbb{F}_{q^2})$ such that $\cX'$ is a curve of degree $q+1$ lying on a non-singular Hermitian variety $\cH_{r',q}$. If this occurs then, by a slight abuse of terminology, $\cX'$ is considered as the (geometric) Frobenius embedding of $\cC$. 

We recall one more result on the Frobenius linear series of $\mathbb{F}_{q^2}$-maximal curves. As before, we limit ourselves on the relevant case for our purpose, i.e. $r=4$.   
\begin{result}\cite[Proposition 10.6]{HKT}
\label{res25092024}
For an $\mathbb{F}_{q^2}$-rational degree one place $P$ of $\mathbb{F}(\cC)$, the number $q+1$ is an order of $\cX$ at $P$. Moreover, let  $0<m_1(P)<m_2(P)<m_3(P)\le q$ be the non-gaps of $\mathbb{F}_{q^2}(\cC)$ at $P$ smaller than $q+1$. Then the order sequence $(j_0=0,j_1,j_2,j_3,j_4)$ of $\cX$ at $P$ is $j_{4-i}+m_i(P)=q+1$ for $i=0,1,2,3$. Also, for non $\mathbb{F}_{q^2}$-rational place $P$ of $\cC$, $q$ is an order of $\cX$ at $P$. 
\end{result}
As a corollary, the following claim holds.
\begin{result}
For a point $P$ of $\cX$ such that the first four rows of the Wronskian matrix at $P$ are linearly independent over $\mathbb{F}$, the order sequence of $\cX$ at $P$ is either $(0,1,2,3,q+1)$, or $(0,1,2,3,q)$ according as $P$ is an $\mathbb{F}_{q^2}$-rational point of $\cX$ or not. 
\end{result}
To determine the $\mathbb{F}_{q^2}$-rational automorphisms, the following result will be useful.    
\begin{result}\cite[Theorem 3.7]{KT}
 \label{res10102024} Let $\cX$ be the Frobenius embedding of an $\mathbb{F}_{q^2}$-maximal curve in $PG(r,\mathbb{F}_{q^2})$. Then $\aut({\mathbb{F}_{q^2}}(\cX))$ is a linear group, the subgroup of $PGU(r+1,\mathbb{F}_{q^2})$ which preserves $\cX$ where $PGU(r+1,\mathbb{F}_{q^2})$ is the automorphism group of the non-singular Hermitian variety $\cH_{r,q}$ containing $\cX$.
\end{result}

\begin{section}{Construction, inherited automorphism group and genus} 
From now on $q\equiv 1 \pmod{4}$ is assumed. 

Fix a fourth primitive root of unity $\lambda$ in $\mathbb{F}_{q^2}$. Then
\begin{equation*}\label{aggiustoesponenti}
\lambda^{q+1}=\lambda^{q-1}\lambda^2=\left(\lambda^4\right)^{\nicefrac{(q-1)}{4}} \lambda^2=\lambda^2.
\end{equation*}
From Result  \ref{struct}, 
$\varphi_\lambda:\left(x,y\right)\mapsto\left(\lambda x, \lambda^2 y\right)$ 
is an $\mathbb{F}_{q^2}$-automorphism of $\cH_q$ fixing both points $O$ and $Y_\infty$. 

Result \ref{gsxA} provides an equation for the quotient curve $\cF$ of the Hermitian curve $\mathcal{H}_q$ with respect to the automorphism group $\Phi=\left\langle\varphi_\lambda\right\rangle$ of order 4. We begin by giving an alternative equation for it. 
\begin{lemma}
\label{lem25092024} The plane curve $\cF$ of equation 
$$\eta^q+\eta \tau^{\nicefrac{(q-1)}{2}}-\tau^{q-\nicefrac{(q-1)}{4}}=0$$
is a quotient curve of $\cH_q$ with respect to $\Phi$. 
\end{lemma}
\begin{proof}
Let $F$ be the fixed field of $\Phi$ in $\mathbb{F}_{q^2}(\cH_q)$. From Galois theory, $[\mathbb{F}_{q^2}(\cH_q):F]=4$. Let
$\xi=x^4$ and 
$\tau=\frac{x^2}{y}.$
Since $$\varphi_\lambda\left(\xi\right)=\varphi_\lambda\left(x^4\right)=\varphi_\lambda\left(x\right)^4=\left(\lambda x\right)^4=x^4=\xi$$ and
$$\varphi_\lambda\left(\tau\right)=\varphi_\lambda\left (\frac{x^2}{y}\right)=\frac{\varphi_\lambda(x)^2}{\varphi_\lambda\left(y\right)}=\frac{\left(\lambda x\right)^2}{\lambda^2 y}=\frac{x^2}{y}=\tau$$
we have $\Fs(\xi,\tau)\subseteq F$. Moreover $\Fs (x,y)=\Fs(x,\tau)$ and  $[\Fs(x,y):\Fs(\xi,\tau)]\leq 4$. As  $[\mathbb{F}_{q^2}(\cH_q):F]=4$, the claim follows. Thus $\mathbb{F}_{q^2}(\cF)=\mathbb{F}_{q^2}(\xi,\tau)$. 
Now, eliminate $x$ and $y$ from the equation of $\cH_q$. Replacing $y$ by $\frac{x^2}{\tau}$ in $y^q+y-x^{q+1}$, 
\begin{equation*}
\left(\frac{x^2}{\tau}\right)^q+\frac{x^2}{\tau}=x^{q+1}. 
\end{equation*}
Then division by $x^2$ gives $$\frac{x^{2q-2}}{\tau^q}+\frac{1}{\tau}=x^{q-1},$$
and $4\mid (q-1)$ yields
\begin{equation}\label{quoziente}
    \xi^{\nicefrac{(q-1)}{2}}+\tau^{q-1}=\xi^{\nicefrac{(q-1)}{4}}\tau^q.
\end{equation}
Let $\eta=\frac{\xi}{\tau}$. Then $\tau=\xi\eta$, and  the birational transformation $\left(\xi,\tau\right)\mapsto \left(\xi,\frac{\xi}{\eta}\right)$ takes (\ref{quoziente}) to 
\begin{equation}
    \label{quozientebella}
    \eta^q+\eta\xi^{\nicefrac{(q-1)}{2}}=\xi^{q-\nicefrac{(q-1)}{4}}.
\end{equation}
which shows that $\cF$ is a quotient curve of $\cH_q$ with respect to $\Phi$. 
\end{proof}
We exhibit a birational map from $\cF$ to the curve in Result \ref{gsxA}.
Let $t=\frac{\xi^{\nicefrac{(q-1)}{2}}}{\tau^{q-1}}$ and $z=\xi$. Substitution in $z^{\nicefrac{(q^2-1)}{4}}=t(t+1)^{q-1}$  gives 
\begin{equation*}
\xi^{\nicefrac{(q^2-1)}{4}}=\frac{\xi^{\nicefrac{(q-1)}{2}}}{\tau^{q-1}}\left ( \frac{\xi^{\nicefrac{(q-1)}{2}}}{\tau^{q-1}}+1 \right )^{q-1}.
\end{equation*}
Therefore, 
\begin{equation*}
\tau^{q-1}\xi^{\nicefrac{(q^2-1)}{4}}=\left ( \frac{\xi^{\nicefrac{(q-1)}{2}+\nicefrac{1}{2}}}{\tau^{q-1}}+\xi^{\nicefrac{1}{2}}\right)^{q-1},
\end{equation*}
whence
\begin{equation*}
\tau\xi^{\nicefrac{(q+1)}{4}}=\frac{\xi^{\nicefrac{q}{2}}}{\tau^{q-1}}+\xi^{\nicefrac{1}{2}}.
\end{equation*}
Thus 
\begin{equation*}
\tau^q\xi^{\nicefrac{(q+1)}{4}-\nicefrac{1}{2}}=\xi^{\nicefrac{q}{2}-\nicefrac{1}{2}}+\tau^{q-1},
\end{equation*}
which gives (\ref{quoziente}).
\begin{lemma}\label{lem16102024}   
Let for a fourth primitive root of unity $\lambda$ in $\mathbb{F}_{q^2}$,
\begin{equation} \label{F1ab}
    \begin{array}{cc} 
    F_1(X,Y)=Y^{q}+YX^{\nicefrac{(q-1)}{2}}-X^{q-\nicefrac{(q-1)}{4}}, & F_2(X,Y)=Y^{q}- YX^{\nicefrac{(q-1)}{2}}+X^{q-\nicefrac{(q-1)}{4}}, \\
    F_3(X,Y)=Y^{q}+\lambda YX^{\nicefrac{(q-1)}{2}}- \lambda X^{q-\nicefrac{(q-1)}{4}}, & F_4(X,Y)=Y^{q}-\lambda YX^{\nicefrac{(q-1)}{2}}+\lambda X^{q-\nicefrac{(q-1)}{4}}.
\end{array}
\end{equation}

Then the plane curves $\cF_i$ of affine equations $F_i(X,Y)=0$ are pairwise isomorphic over $\mathbb{F}_{q^2}$. 
\end{lemma}
\begin{proof} First we show that $\cF_2$ is birationally isomorphic to $\cF_1$ over $\mathbb{F}_{q^2}$. For a primitive element $w$ of $\mathbb{F}_{q^2}$, let $u=w^2$. Then $u^{\nicefrac{(q^2-1)}{4}}=-1$. Let $\mu$ be the affine transformation $(X,Y)\mapsto (uX,vY)$ where $v=u^{\nicefrac{(q+3)}{4}}$. We show that $\mu$ takes $\cF_1$ to $\cF_2$. Since
$$(vY)^q+vY(uX)^{\nicefrac{(q-1)}{2}}-(uX)^{q-\nicefrac{(q-1)}{4}}=v^q(Y^q+\frac{1}{v^{q-1}}u^{\nicefrac{(q-1)}{2}}YX^{\nicefrac{(q-1)}{2}}-\frac{u^{q-\nicefrac{(q-1)}{4}}}{v^q}X^{q-\nicefrac{(q-1)}{4}})$$
whence the claim follows. Similarly, $\mu$ takes $\cF_3$ to $\cF_4$. It remains to show that either $\cF_3$, or $\cF_4$ is birationally isomorphic to $\cF_1$ over $\mathbb{F}_{q^2}$. Since $w$ is a primitive element of $\mathbb{F}_{q^2}$, we have either $\lambda=w^{\nicefrac{(q^2-1)}{4}}$, or $\lambda^{-1}=-\lambda=w^{\nicefrac{(q^2-1)}{4}}$.   
Let $\delta$ be  the affine transformation $(X,Y)\mapsto (wX,vY)$ where $v=w^{\nicefrac{(q+3)}{4}}$. Since
$$(vY)^q+vY(wX)^{\nicefrac{(q-1)}{2}}-(wX)^{q-\nicefrac{(q-1)}{4}}=v^q\left (Y^q+\frac{1}{v^{q-1}}w^{\nicefrac{(q-1)}{2}}YX^{\nicefrac{(q-1)}{2}}-\frac{w^{q-\nicefrac{(q-1)}{4}}}{v^q}X^{q-\nicefrac{(q-1)}{4}} \right)$$
$$=v^q\left (Y^q-w^{\nicefrac{(q^2-1)}{4}}YX^{\nicefrac{(q-1)}{2}}+w^{\nicefrac{(q^2-1)}{4}}X^{q-\nicefrac{(q-1)}{4}} \right).$$
This shows that $\delta$ takes $\cF_1$ to either $\cF_3$, or $\cF_4$ according as either $\lambda=-w^{\nicefrac{(q^2-1)}{4}}$, or $\lambda=w^{\nicefrac{(q^2-1)}{4}}$.    

\end{proof}

\begin{lemma}\label{automorfismiquoziente}
    The $\mathbb{F}_{q^2}$-automorphism group of the curve of equation (\ref{quozientebella}) has a subgroup of order $\frac{1}{4}(q^2-1)$ generated by $\Psi:(\xi,\eta)\mapsto (\theta\xi,\zeta\eta),$ with $\theta^{\nicefrac{(q^2-1)}{4}}=1$ and $\zeta=\theta^{\nicefrac{(q+3)}{4}}$.
\end{lemma}
\begin{proof}
The map $\Psi$ is an $\mathbb{F}_{q^2}$-automorphism of the curve, since
$\zeta^q\eta^q+\zeta \eta \theta^{\nicefrac{(q-1)}{2}}\xi^{\nicefrac{(q-1)}{2}}=\theta^{q-\nicefrac{(q-1)}{4}}\xi^{q-\nicefrac{(q-1)}{4}}.$
Therefore
 $\zeta^q=\zeta \theta^{\nicefrac{(q-1)}{2}}=\theta^{q-\nicefrac{(q-1)}{4}},$
whence
$\zeta^{q-1}=\theta^{\nicefrac{(q-1)}{2}}.$
Thus
    $\zeta=\theta^{q-\nicefrac{(q-1)}{4}-\nicefrac{(q-1)}{2}}=\theta^{\nicefrac{(q+3)}{4}}$
and
    $\theta^{\nicefrac{(q^2-1)}{4}}=1.$
\end{proof}
\begin{lemma}\label{automorfismiquoziente2}
    The $\mathbb{F}_{q^2}$-automorphism group of the curve of equation (\ref{quozientebella}) has a subgroup of order $2$ generated by $\Theta:(\xi,\eta)\mapsto \left(\frac{\xi^3}{\eta^4},\frac{\xi^2}{\eta^3}\right)$.
\end{lemma}
\begin{proof}
The map $\Theta$ has order $2$, since
$\Theta^2(\xi,\eta)=\Theta(\Theta(\xi,\eta))=\Theta\left(\frac{\xi^3}{\eta^4},\frac{\xi^2}{\eta^3}\right)=\left(\xi,\eta\right).$
 From
\begin{equation*}
\Theta\left(\eta^q+\eta\xi^{\nicefrac{(q-1)}{2}}-\xi^{q-\nicefrac{(q-1)}{4}}\right)=0
\end{equation*}
we have
\begin{equation*}
    \Theta(\eta)^q+\Theta(\eta)\Theta(\xi)^{\nicefrac{(q-1)}{2}}-\Theta(\xi)^{q-\nicefrac{(q-1)}{4}}=0.
\end{equation*}
Therefore
\begin{equation*}
    \left(\frac{\xi^2}{\eta^3}\right)^q+\left(\frac{\xi^2}{\eta^3}\right)\left(\frac{\xi^3}{\eta^4}\right)^{\nicefrac{(q-1)}{2}}-\left(\frac{\xi^3}{\eta^4}\right)^{q-\nicefrac{(q-1)}{4}}=0,
\end{equation*}
and the claim follows multiplying both sides by $\eta^{3q+1}/\xi^{\nicefrac{(3q+1)}{2}}$.
\end{proof}
The automorphism group $\Phi$ fixes both points $O$ and $Y_\infty$ of $\mathcal{H}_q$. Furthermore, under the action of $\Phi$, the remaining $q-1$ points on the line $x=0$ form $\ha(q-1)$ orbits each of length $2$, while the points off that line form long orbits. Therefore, the Hurwitz genus formula applied to $\Phi$, see \cite[Theorem 11.57]{HKT}, gives the following result.  
\begin{lemma}
    The genus of the curve of equation (\ref{quozientebella}) is equal to $\frac{1}{8}(q-1)^2.$
\end{lemma}
\section{Branches centered in the singular point }
\label{secCAG}
The plane curve $\cF$ of equation (\ref{quozientebella}) has two singular points, the origin $O=(0,0)$ and the only point at infinity $P_{\infty}=(1:0:0)$.  
Since the origin $O$ has multiplicity equal to $\ha(q-1)+1=\ha(q+1)$, there exist at most $\ha(q+1)$ branches centered at $O$; see \cite[Theorem 4.36 (iii)]{HKT}. Actually, this bound is attained. In fact, from $\xi=x^4$, the branches of $\cF$ centered at $O$ arise from those $\Phi$-orbits on $\cH_q$ which lie over the point $(0,0)$ in the cover $\cH_q|\cF$. Such a $\Phi$-orbit either consists of  $(0,0)$ or it has size two, namely one of the pairs $\{(0,b),(0,-b)\}$ where $b^{q-1}=1$. Altogether, their number is $\ha(q+1)$. Furthermore, each of them is linear, see  \cite[Theorem 4.36 (ii)]{HKT}. From Equation (\ref{quozientebella}), the tangent lines to $\cF$ at $O$ are the $\xi$ and $\eta$ axes. The point $P_\infty$ of the $\eta$-axis is the center of unique branch. In fact, that branch arises from the branch of $\mathcal{H}_q$ centered at the point $(0:1:0)$ which is fixed by $\Phi$.  
A straightforward computation proves the following lemma. 
\begin{lemma} \label{rami}
    A primitive representation of a branch $P_c$ of the curve (\ref{quozientebella}) centered at the origin $O$ and tangent to the $\eta$-axis is
    \begin{equation} \label{ramoorgineY}
        \left(\xi = ct^2+ 2ct^{\nicefrac{(q+1)}{2}}+\cdots, \eta = t\right),
    \end{equation}
    with $c\in\mathbb{F}_{q^2}^{\ast}$, $c^{\nicefrac{(q-1)}{2}}=-1$.
    A primitive representation of a branch $P_0$ of the curve (\ref{quozientebella}) centered at the origin $O$ and tangent to the $\xi$-axis is
    \begin{equation} \label{ramoorgineX}
        \left(\xi = t , \eta = t^{\nicefrac{(q+3)}{4}}+\cdots\right).
    \end{equation}
A primitive representation of a branch of the curve (\ref{quozientebella}) centered at the point at infinity $P_\infty$ is
\begin{equation} \label{ramoinf}
\left(\xi=\frac{1}{t^q},\eta=\frac{t^{\nicefrac{(q-1)}{4}}}{t^q}+\cdots\right).
\end{equation}
By a little abuse of notation, we denote that branch by the same letter $P_\infty$.
\end{lemma}
As a consequence of (\ref{ramoinf}), $P_\infty$ is a $\frac{1}{4}(q-1)$-fold point of $\cF$ and the tangent line is the line at infinity. Moreover, in terms of \cite[Section 4.2]{HKT}, (\ref{ramoorgineY}) shows that each branch $P_c$ of $\cF$ centered at $O$ and tangent to the $\eta$-axis has order sequence $(0,1,2)$ while, from (\ref{ramoorgineX}), $(0,1,\frac{1}{4}(q+3))$ is the order sequence of each branch $P_0$ of $\cF$ centered at $O$ and tangent to the $\xi$-axis. If the former branches count $r$ and the latter ones $m$, two equations are obtained: $2r+m=q$ as the point of infinity of the $\xi$-axis is not a point of $\cF$; $r+\frac{1}{4}(q+3)m+\frac{1}{4}(q-1)=q$ as the unique points of $\cF$ on the $\xi$-axis are $O$ and $P_\infty$. From these diophantine equations, $r=\ha(q-1)$ and $m=1$.


The image of a branch $P_c$ centered at the origin $O$ with tangent the $\eta$-axis under the automorphism $\Psi$ defined in Lemma \ref{automorfismiquoziente} is the  branch with primitive representation
\begin{equation*}
    \left(\xi=\theta c t^2+\cdots, \eta=\theta^{\nicefrac{(q+3)}{4}}t\right)
\end{equation*}
centered at the origin.
Replacing $\theta^{(q+3)/4}t$ by $t^{\ast}$, an equivalent representation is 
\begin{equation} \label{15}
    \left(\xi=\frac{1}{\theta^{\nicefrac{(q+1)}{2}}}ct^{\ast 2}+\cdots, \eta=t^{\ast}\right).
\end{equation}
Since $\left(\frac{1}{\theta^{\nicefrac{(q+1)}{2}}}c\right)^{\nicefrac{(q-1)}{2}}=-1$, replacing $\frac{1}{\theta^{\nicefrac{(q+1)}{2}}}c$ by $c'$, the branch represented by (\ref{15}) is the branch centered at the origin and tangent to the $\eta$-axis with representation
\begin{equation*}
    \left(\xi=c't^{\ast 2}+\cdots, \eta=t^{\ast}\right).
\end{equation*}
The image under the automorphism $\Theta$ defined in Lemma \ref{automorfismiquoziente2} is the branch with primitive representation
\begin{equation*}
\left(\xi=c^3 t^2+\cdots, \eta=c^2 t+\cdots\right)
\end{equation*}
centered at the origin.
Replacing $c^2t+\cdots$ by $t^{\ast}$, an equivalent representation is 
\begin{equation}\label{16}
\left(\xi=\frac{1}{c} t^{\ast 2}+\cdots, \eta=t^{\ast}\right).
\end{equation}
Since $\left(\frac{1}{c}\right)^{\nicefrac{(q-1)}{2}}=-1$, replacing $\frac{1}{c}$ by $c'$, the branch represented by (\ref{16}) is the branch centered at the origin and tangent to the $\eta$-axis with representation
\begin{equation*}
    \left(\xi=c't^{\ast 2}+\cdots, \eta=t^{\ast}\right).
\end{equation*}
The image of the branch $P_0$ centered at the origin $O$ with tangent the $\xi$-axis under the automorphism $\Psi$ defined in Lemma \ref{automorfismiquoziente} is the  branch with primitive representation
\begin{equation*}
\left(\xi=\theta t, \eta=\theta^{\nicefrac{(q+3)}{4}}t^{\nicefrac{(q+3)}{4}}+\cdots\right)
\end{equation*}
centered at the origin.
Replacing $\theta t$ by $t^{\ast}$, an equivalent representation is
\begin{equation*}
    \left(\xi=t^{\ast}, \eta=t^{\ast\nicefrac{(q+3)}{4}}+\cdots\right).
\end{equation*}
The image under the automorphism $\Theta$ defined in Lemma \ref{automorfismiquoziente2} is the branch with primitive representation
\begin{equation*}
    \left(\xi=\frac{1}{t^q},\eta=\frac{t^{\nicefrac{(q-1)}{4}}}{t^q}+\cdots\right)
\end{equation*}
centered at the point at infinity $P_\infty$.

The image of the branch centered at the point at infinity $P_\infty$ under the automorphism $\Psi$ defined in Lemma \ref{automorfismiquoziente} is a branch with primitive representation
\begin{equation*}
\left(\xi=\theta\frac{1}{t^q}, \eta=\theta^{\nicefrac{(q+3)}{4}}\frac{t^{\nicefrac{(q-1)}{4}}}{t^q}+\cdots\right)
\end{equation*}
centered at the point at infinity.
An equivalent representation is
\begin{equation*}
\left(\xi=\frac{1}{t^{\ast q}}, \eta=\frac{t^{\ast\nicefrac{(q-1)}{4}}}{t^{\ast q}}+\cdots\right).
\end{equation*}

The image under the automorphism $\Theta$ defined in Lemma \ref{automorfismiquoziente2} is the branch with primitive representation
\begin{equation*}
    \left(\xi=t, \eta=t^{\nicefrac{(q+3)}{4}}+\cdots\right)
\end{equation*}
centered at the origin.

\section{Weierstrass semigroup and natural embedding}
As we have seen in Section \ref{secCAG}, $P_\infty$ is the center of a unique branch of $\cF$. By a little abuse of terminology, $P_\infty$ also stands for that branch. Let $H\left(P_\infty\right)$ be the Weierstrass semigroup of $\cF$ at $P_\infty$.
\begin{lemma}
\label{lemWS} $\ha (q+1),q-\frac{1}{4}(q-1),q\in H\left(P_\infty\right)$. 
\end{lemma}
\begin{proof} To prove that $\ha (q+1)\in H\left(P_\infty\right)$, it is enough to show that $P_\infty$ is the only pole of $\eta^2/\xi$ and that the relative pole number is $\ha (q+1)$. According to Section \ref{secCAG}, let $\gamma_1,\ldots,\gamma_{(q-1)/2}$ be the branches of $\cF$ centered at $O$ and tangent to the $\eta$-axis, and $\gamma_0$ the branch of $\cF$ centered at $O$ and tangent to $\xi$-axis. 
By Lemma \ref{rami}, B\'ezout's theorem together with \cite[Theorem  4.36 (ii)]{HKT} yield the following equation on the intersection divisor $\cF\circ \ell_1$ where $\ell_1$ is the line of equation $\eta=0$, i.e. the $\xi$-axis: 
\begin{equation}\label{div1}
    \cF\circ \ell_1 =\sum I\left(P,\ell\cap\gamma\right)\gamma = \sum_{i=1}^{(q-1)/2} \gamma_i+\frac{1}{4}(q+3)\gamma_0+\frac{1}{4}(q-1)P_\infty.
\end{equation}
Also, let $\ell_2$ denote the line of equation $\xi=0$, i.e. the $\eta$-axis. Similarly, from Lemma \ref{rami} together with  B\'ezout's theorem and \cite[Theorem  4.36 (ii)]{HKT}, 
\begin{equation}\label{div2}
    \cF\circ \ell_2 =\sum I\left(P,\ell\cap\gamma\right)\gamma = \sum_{i=1}^{(q-1)/2} 2\gamma_i+\gamma_0.
\end{equation}
Moreover, let $\ell_\infty$ denote the line at infinity. Then 
\begin{equation}\label{div3}
    \cF\circ \ell_\infty=qP_\infty.
\end{equation}
From (\ref{div1}), (\ref{div2}), (\ref{div3}) and \cite[Theorem 6.42]{HKT}, 

$$\div\left(\frac{\eta^2}{\xi}\right)=2\,\div(\eta)-\div(\xi)=
2\left(\sum_{i=1}^{(q-1)/2} \gamma_i+\frac{1}{4}(q+3)\gamma_0+\frac{1}{4}(q-1)P_\infty-qP_\infty\right)-\left(\sum_{i=1}^{(q-1)/2} 2\gamma_i+\gamma_0-qP_\infty\right).$$
From this,
$$\div\left(\frac{\eta^2}{\xi}\right)=\ha(q+1)\gamma_0-\ha(q+1)P_\infty$$
which shows that $\ha(q+1)$ belongs indeed to $H\left(P_\infty\right)$. 

From (\ref{div1}), (\ref{div3}) and \cite[Theorem 6.42]{HKT},
$$\div(\eta)= \sum_{i=1}^{(q-1)/2} \gamma_i+\frac{1}{4}(q+3)\gamma_0+\frac{1}{4}(q-1)P_\infty-qP_\infty= \sum_{i=1}^{(q-1)/2} \gamma_i+\frac{1}{4}(q+3)\gamma_0-(q-\frac{1}{4}(q-1))P_\infty$$
whence $q-\frac{1}{4}(q-1)\in H\left(P_\infty\right)$.

Moreover, since $\cF$ is an $\mathbb{F}_{q^2}$-maximal curve, \cite[Proposition 10.2]{HKT} yields $q\in H\left(P_\infty\right)$. 
\end{proof}
\begin{rem}
{\em{It seems plausible that $H\left(P_\infty\right)$ is generated by the integers in Lemma \ref{lemWS}.}}
\end{rem}
\begin{theorem}
\label{lemWS1} The numbers in $H\left(P_\infty\right)$ in the interval $[1,q]$ are exactly three, namely  $\ha(q+1),q-\frac{1}{4}(q-1)$ and $q$.
\end{theorem}
\begin{proof} Let $\omega\in H\left(P_\infty\right)$ and take $u\in\mathbb{F}_{q^2}(\xi,\eta)$ such that $\div(u)=\omega P_\infty$. Observe that $u$ can also be viewed as an element of $\mathbb{F}_{q^2}(\cH_q)$.
Since the cover $\cH_q|\cF$ is totally ramified at $P_\infty$, $u$ has also a unique pole in $\mathbb{F}_{q^2}(\cH_q)$. This pole is the branch $\delta$ of $\cH_q$ centered at the unique point  of $\cH_q$ at infinity. Since the cover $\cH_q|\cF$ has degree $4$, \cite[Lemma 7.19]{HKT} yields
\begin{equation*}
  v_\delta(u)=\text{ord}_{\delta}(u) =4 \,\text{ord}_{P_\infty}(u)=4\,v_{P_\infty}(u).
\end{equation*}
As the Weierstrass semigroup of the Hermitian curve at any $\mathbb{F}_{q^2}$-rational point is generated by $q$ and $q+1$, for each non-gap $m$ at the point $P_{\infty}$ there exist positive integers $\alpha,\beta$ such that $4m=\alpha q+\beta (q+1)$. For $m\le q$ this is only possible for $\ha(q+1),q-\frac{1}{4}(q-1)$ and $q$. 
\end{proof}
We prove the analog of Theorem \ref{lemWS1} for the branches centered at $O$ and tangent to the $\eta$-axis. 
\begin{proposition}
\label{pro26092024B} At each of the $\ha(q-1)$ branches centered at $O$ and tangent to the $\eta$-axis, $\ha(q-3)$ is a non-gap.
\end{proposition} 
\begin{proof} We adopt the notation from Lemma \ref{rami}. For $c\in \mathbb{F}_{q^2}^{\ast}$ with $c^{\nicefrac{(q-1)}{2}}=-1$ let $P_c$ be the branch with primitive representation as in (\ref{ramoorgineY}). 
It is enough to verify that the pole divisor of the function $$\varphi=\frac{\eta^2}{\xi-c\eta^2}$$ is equal to $\ha(q+1)P_c$. Let $P(a,b)$ be an affine point of $\cF$ other than the origin. If $P(a,b)$ is a pole of $\varphi$ then $a-cb^2=0$. Since $P(a,b)$ is a point of $\cF$ and $c^{\nicefrac{(q-1)}{2}}=-1$, this yields 
$$b^q+a^{\nicefrac{(q-1)}{2}}-a^{q-\nicefrac{(q-1)}{4}}=b^q-b^q-a^{q-\nicefrac{(q-1)}{4}}=0$$
whence $a=0$. As $a=0$ implies $b=0$, this contradicts our hypothesis. Next we verify that if $d\in \mathbb{F}_{q^2}^{\ast}$ with $d^{\nicefrac{(q-1)}{2}}=-1$, then the branch $P_d$ centered at $O$ and tangent to the $\eta$-axis is a pole of $\varphi$ if and only if $c=d$. From (\ref{ramoorgineY}), $\varphi$ evaluated at $P_d$
gives 
$$v_{P_d}(\varphi)=v_{P_d}\left(\frac{\eta^2}{\xi-c\eta^2}\right)=\ord_t\left(\frac{t^2}{(d-c)t^2+2dt^{\nicefrac{(q+1)}{2}}+\ldots}\right)=\ord_t\left(\frac{1}{(d-c)+2dt^{\nicefrac{(q-3)}{2}}+\ldots}\right).$$
Thus $\varphi$ has a pole at $P_d$ only for $c=d$, and if this is the case then the pole number is $\ha(q-3)$. 
Next we verify that the branch $P_0$ of $\cF$ centered at $O$ and tangent to the $\xi$-axis is not a pole of $\varphi$. From (\ref{ramoorgineX}), $\varphi$ evaluated at $P_0$ gives 
$$v_{P_0}(\varphi)=v_{P_0}\left(\frac{\eta^2}{\xi-c\eta^2}\right)=
\ord_t\left(\frac{(t^{(q+3)/4}+\ldots)^2}{t-c(t^{(q+3)/4}+\ldots)^2}\right).$$
Therefore $P_0$ is a zero of $\varphi$.

It remains to verify that $P_\infty$ is not a pole of $\varphi$. Write 
$$\varphi=\frac{\eta^2}{\xi-c\eta^2}=\frac{\eta^2}{\xi-c\eta^2}+\frac{1}{c}-\frac{1}{c}
=\frac{1}{c-c^2 \frac{\eta^2}{\xi}}-\frac{1}{c}.$$
As we have seen in the proof of Lemma \ref{lemWS}, $P_\infty$ is a pole of $\eta^2/\xi$. Therefore, $P_\infty$ is a zero of $1/(c-c^2 \eta^2/\xi)$, and hence $\varphi$ evaluated at $P_\infty$ is equal to $-1/c$. This completes the proof.  
\end{proof}
Theorem \ref{lemWS1} and \cite[Section 10.2]{HKT} have the following corollary.
\begin{theorem}
The (projective) Frobenius dimension of $\cF$ is equal four.  
\end{theorem}
From the proof of Theorem \ref{lemWS1}, the Riemann-Roch space of the Frobenius divisor $\mathcal{D}=(q+1)P_\infty$ is 
\begin{equation*}
\mathcal{L}(\mathcal{D})=\Big\langle 1,\xi,\eta,\frac{\eta^2}{\xi},\left(\frac{\eta^2}{\xi}\right)^2\Big\rangle,
\end{equation*}
and the corresponding natural embedding of $\cF$ in the four-dimensional projective space $PG(4,\mathbb{F}_{q^2})=PG(4,q^2)$, see \cite[Section 10.2]{HKT}, is given by 
\begin{equation}
    \label{coordinate}
   \rho:\, (a,b)\mapsto\left (1,a,b,\frac{b^2}{a},\left(\frac{b^2}{a}\right)^2\right),
\end{equation}
or, in homogeneous coordinates,
\begin{equation*} 
    \rho:\, \left(a:b:c\right)\mapsto\left (a^2c^2:a^3c:a^2bc:ab^2c:b^4\right).
\end{equation*}
Let $\cX$ be the (geometrically irreducible, non-singular) curve of degree $q+1$ obtained by the above natural embedding of $\cF$ in $PG\left(4,q^2\right)$. As usual, $\cX$ is viewed as a curve of $PG\left(4,\mathbb{F}\right)$. 
\begin{theorem}
\label{thene} The curve $\cX$ is in the intersection of a Hermitian variety and two quadrics.  
\end{theorem}
\begin{proof}
Fix a homogeneous coordinate system $\left(Y_0:Y_1:Y_2\right)$ in $PG\left(2,\mathbb{F}\right)$ and a homogeneous coordinate system $\left(X_0:X_1:X_2:X_3:X_4\right)$ in $PG\left(4,\mathbb{F}\right)$, and define the map $\sigma: PG\left(2,\mathbb{F}\right)\mapsto PG\left(4,\mathbb{F}\right)$ by
$$X_0=Y_0^2Y_2^2,\, X_1=Y_0^3Y_2,\, X_2=Y_0^2Y_1Y_2,\,X_3=Y_0Y_1^2Y_2\,X_4=Y_1^4.$$
Then $\cX$ is contained in the Hermitian variety $\cH_{4,q}$ of equation
\begin{equation*} 
     -X_1^{q+1}+4X_2^{q+1}-2X_3^{q+1}+X_0X_4^q+X_0^qX_4=0.
\end{equation*}
In fact, if $P=(1:a:b)$ is a generic point of $\cF$ in $PG\left(2,\mathbb{F}\right)$ then its image by $\sigma$ is the point $P'=\left(a^2:a^3:a^2b:ab^2:b^4\right)$ in $PG\left(4,\mathbb{F}\right)$. This point $P'$ lies in the Hermitian variety $\cH_{4,q}$ if and only if 
\begin{equation}
    \label{Fab}
    F(a,b)=-a^{3q+1}+4a^{2q}b^{q+1}-2a^{q-1}b^{2q+2}+b^{4q}+a^{2q-2}b^4=0.
\end{equation}
A straightforward computation shows that $F(a,b)$ is equal to $F_1(a,b)F_2(a,b)F_3(a,b)F_4(a,b)$ where $F_1(X,Y)$, $F_2(X,Y),F_3(X,Y),F_4(X,Y)$ are defined in (\ref{F1ab}).

    
From (\ref{quozientebella}) the assertion follows.

Furthermore, $\cX$ is contained in both quadrics $\mathcal{Q}_1$ and  $\mathcal{Q}_2$ of equations respectively
\begin{equation}
\label{quadric1}
X_3^2-X_0X_4=0,
\end{equation}
and 
\begin{equation}
\label{quadric2}
 X_2^2-X_1X_3=0.
\end{equation}
\end{proof}
\begin{rem}
\label{rem16102024A}
\em{
We point out that $\cX$ is a component of the complete intersection of $\cH_{4,q}$ with $Q_1$ and $Q_2$. Take a point $R'=(u_0:u_1:u_2:u_3:u_4)\in PG(4,\mathbb{F})$ lying on $\cH_{4,q}$ and both quadrics $\mathcal{Q}_1$ and $\mathcal{Q}_2$ which is not the image of a point of $\cF$ by the map $\sigma$. 
Such a point $R$ is different from $P_0,P_\infty,P_c$ with $c^{\nicefrac{(q-1)}{2}}=-1$. In particular, $u_0\ne 0$. Therefore, $u_0=1$ may be assumed. Also, $u_1\neq 0$ as $R'\ne P_c$. Let $a=u_1$, $b=u_2$ and $R=(1:a:b)$. We show that $\sigma$ takes $R$ to $R'$. Since
$b^2=a u_3$ and $u_3^2=u_4$, we have 
$$\sigma(R)=\left(1:a:b:\frac{b^2}{a}:\left(\frac{b^2}{{a}}\right)^2\right)=(1:a:b:u_3:u_4).$$  
Since $R'\in \cH_{4,q}$, we also have $-a^{q+1}+4b^{q+1}-2u_3^{q+1}+u_4^q+u_4=0$. Replacing $u_3$ by $b^2/a$ and $u_4$ by $u_3^2=(b^2/a)^2$, this gives (\ref{Fab}). As we have shown before, $F(a,b)$ splits into four irreducible factors, $F_i(a,b)$; see (\ref{F1ab}). Let $\cF_i$ be the plane curve of equation $F_i(X,Y)=0$, in particular $\cF_1=\cF$. The complete intersection of $\cH_{4,q}$ with $Q_1$ and $Q_2$ splits into four curves of degree $q+1$. Each of them is an $\mathbb{F}_{q^2}$-maximal curve which are pairwise isomorphic over $\mathbb{F}_{q^2}$ by Lemma \ref{lem16102024}. }
\end{rem}
\begin{proposition}
\label{propoints} The map $\rho$ takes the branch $P_c$ in (\ref{ramoorgineY}) to the point $Q_c=\left(c^2:0:0:c:1\right)$ of $\cX$,
the branch $P_0$ in (\ref{ramoorgineX}) to the point $Q_0=\left(1:0:0:0:0\right)$ of $\cX$ and the branch $P_\infty$ in (\ref{ramoinf}) to the point $Q_\infty=\left(0:0:0:0:1\right)$ of $\cX$.
\end{proposition}
\begin{proof} The image of a branch $P_c$ centered at the origin $O$ with tangent the $\eta$-axis (\ref{ramoorgineY}) under $\rho$ is the  branch with primitive branch representation 
\begin{equation*}
    \left (1:ct^2+2ct^{\nicefrac{(q+1)}{2}}+\cdots:t:\frac{1}{c+2ct^{\nicefrac{(q-3)}{2}}+\cdots}:\left(\frac{1}{c+2ct^{\nicefrac{(q-3)}{2}}+\cdots}\right)^2\right)
\end{equation*}
centered at the point $\left(1:0:0:\frac{1}{c}:\frac{1}{c^2}\right)$.
The image of the branch $P_0$ centered at the origin $O$ with tangent the $\xi$-axis (\ref{ramoorgineX}) under $\rho$ is the  branch with primitive branch representation 
\begin{equation*}
    \left(1:t:t^{\nicefrac{(q+3)}{4}}+\cdots:t^{\nicefrac{(q+1)}{2}}+\cdots:t^{q+1}+\cdots\right)
\end{equation*}
centered at the point $(1:0:0:0:0)$.
The image of the branch centered at the point at infinity $P_{\infty}$ (\ref{ramoinf}) under $\rho$ is the branch with primitive branch representation 
\begin{equation}\label{ramoimmagineinf}
    \left(1:\frac{1}{t^q}:\frac{t^{\nicefrac{(q-1)}{4}}}{t^q}+\cdots:\frac{t^{\nicefrac{(q-1)}{2}}}{t^q}+\cdots:\left(\frac{t^{\nicefrac{(q-1)}{2}}}{t^q}+\cdots\right)^2\right)
\end{equation}
centered at $(0:0:0:0:1)$.
In fact, (\ref{ramoimmagineinf}) can also be written as 
$$ \left(1:\frac{1}{t^q}:\frac{1}{t^{\nicefrac{(3q+1)}{4}}}+\cdots:\frac{1}{t^{\nicefrac{(q+1)}{2}}}\cdots:\frac{1}{t^{q+1}}+\cdots\right),$$ and hence it is equivalent to
$$\left(t^{q+1}+\cdots:t+\cdots:t^{\nicefrac{(q+3)}{4}}+\cdots:t^{\nicefrac{(q+1)}{2}}+\cdots:1\right). $$
\end{proof}
The automorphisms of $\cX$ have linear action in $PG(4,\mathbb{F})$; see \cite[Section 10.3]{HKT} and \cite{KT}. We determine this action for the automorphisms in Lemmas \ref{automorfismiquoziente} and \ref{automorfismiquoziente2}.
\begin{proposition}
\label{actionpsi} If $\Psi$ is the automorphism of $\cF$ defined in Lemma \ref{automorfismiquoziente}, then 
$$\Psi:\,\left(X_0:X_1:X_2:X_3:X_4\right)\mapsto \left(X_0: \theta X_1: \theta^{\nicefrac{(q+3)}{4}}X_2: \theta^{\nicefrac{(q+1)}{2}}X_3: \theta^{q+1}X_4\right).$$
\end{proposition}
\begin{proof} 
From $\Psi:\left(\xi,\eta\right)\mapsto\left(\theta \xi, \theta^{\nicefrac{(q+3)}{4}}\eta\right)$, 
$$
\Psi\left(\xi\right)=\theta \xi;\,\Psi\left(\eta\right)=\theta^{\nicefrac{(q+3)}{4}}\eta;\,
\Psi\left(\frac{\eta^2}{\xi}\right)=\theta^{\nicefrac{(q+1)}{2}}\frac{\eta^2}{\xi};\,
\Psi\left(\left(\frac{\eta^2}{\xi}\right)^2\right)=\theta^{q+1}\frac{\eta^4}{\xi^2}.
$$
For a generic point $P=(1:a:b)$ of $\cF$, (\ref{coordinate}) implies  $\sigma(P)=\left(1,a,b,\frac{b^2}{a},(\frac{b^2}{a})^2\right)$. Therefore,   
$$\Psi(P)=\left(1:\theta a:\theta^{\nicefrac{(q+3)}{4}} b:\theta^{\nicefrac{(q+1)}{2}} \frac{b^2}{a}:\theta^{q+1}\frac{b^4}{a^2}\right) $$ 
whence the claim follows.
\end{proof}
\begin{proposition} 
If $\Theta$ is the automorphism of $\cF$ defined in Lemma \ref{automorfismiquoziente2}, then 
$$\Theta:\,\left(X_0:X_1:X_2:X_3:X_4\right)\mapsto \left(X_4: X_1:X_2:X_3: X_0\right).$$
\end{proposition}
\begin{proof}
Since $\Theta:\left(\xi,\eta\right)\mapsto\left(\frac{\xi^3}{\eta^4},\frac{\xi^2}{\eta^3}\right)$, 
$$
\Theta\left(\xi\right)=\frac{\xi^3}{\eta^4};\,
\Theta\left(\eta\right)=\frac{\xi^2}{\eta^3};\,
\Theta\left(\frac{\eta^2}{\xi}\right)=\frac{\xi}{\eta^2};\,
\Theta\left(\left(\frac{\eta^2}{\xi}\right)^2\right)=\frac{\xi^2}{\eta^4}.
$$
For a generic point $P=(1:a:b)$ of $\cF$, we have $\rho(P)=\left(1,a,b,\nicefrac{b^2}{a},(\nicefrac{b^2}{a})^2\right)$. Since 
$$\Theta(P)=\left(1:\frac{a^3}{b^4}:\frac{a^2}{b^3}:\frac{a}{b^2}:\frac{a^2}{b^4}\right)=\left(\left(\frac{b^2}{a}\right)^2:a:b:\frac{b^2}{a}:1\right)$$
the claim follows. 

\end{proof}
\begin{proposition}
The group $U$ generated by $\Psi$ and $\Theta$ has order $\ha(q^2-1)$. More precisely,
$$U=<\Psi,\Theta|\Theta \Psi\Theta=\Psi^{-q}>.$$
\end{proposition}
\begin{proof} A straightforward computation shows that 
$$\Theta\Psi\Theta: \left(X_0:X_1:X_2:X_3:X_4\right)\mapsto\left(\theta^{q+1}X_0:\theta X_1:\theta^{(q+3)/4}X_2:\theta^{(q+1)/2}X_3:X_4\right).$$
This together with $\theta^{(q^2-1)/2}=1$ imply  $\Theta\Psi\Theta=\Psi^{-q}$. 
\end{proof}
\begin{proposition}
    \label{pro26092024} Let $\mathcal{O}$ be the set of the points of $\cX$ arising from the branches of $\cF$ centered at singular points, that is, $\mathcal{O}=\mathcal{O}_1\cup \mathcal{O}_2$ where $\mathcal{O}_1=
\{Q_0,Q_\infty\}$ and $\mathcal{O}_2=\{Q_c|c^{\nicefrac{(q-1)}{2}}=-1\}$. Then both $\mathcal{O}_1$ and $\mathcal{O}_2$ are orbits of $\aut(\cX)$. 
\end{proposition}
\begin{proof} Proposition \ref{actionpsi} shows that $\Psi$ fixes both $Q_0$ and $Q_\infty$. Let $d=c/\theta^{\nicefrac{(q+1)}{2}}$. Then $d^{\nicefrac{(q-1)}{2}}=-1$, and Proposition \ref{actionpsi} shows that $\Psi$ takes $Q_c$ to $Q_d$. 
Since there are as many points in $\mathcal{O}_2$ as the order of $\theta^{\nicefrac{(q+1)}{2}}$ in $\mathbb{F}_{q^2}^*$, the group generated by $\Psi$ acts transitively on $\mathcal{O}_2$. Moreover, $\Theta$ 
interchanges $Q_0$ and $Q_\infty$, and takes $Q_c$ to $Q_{c^{-1}}$.  
\end{proof}


\end{section}


\section{The $\mathbb{F}_{q^2}$-rational automorphism group}    
A Magma aided computation allows to determine $\aut({\mathbb{F}_{q^2}}(\cX))$ for the first three values of $q$, namely $q=5,9,13$. For $q=5$, $\aut({\mathbb{F}_{25}}(\cX))$ has order $240$, moreover $\cX$ is $\mathbb{F}_{625}$-isomorphic to the Roquette curve; see \cite[Theorem 11.127 (II)]{HKT}.  For $q=9,13$ instead, $\aut({\mathbb{F}_{q^2}}(\cX))$ has order $\ha(q^2-1)$ and hence it is the inherited automorphism group. We show that the latter claim holds true for $q\ge 17$. 
\begin{theorem}
\label{teo10102024} The $\mathbb{F}_{q^2}$-rational automorphism group $\aut(\mathbb{F}_{q^2}(\cX))$ of the curve $\cX$ 
has order $\ha(q^2-1)$ and hence it is inherited from the automorphism group of the Hermitian curve. 
\end{theorem}
\begin{proof}
We may assume $q\ge 17$. From Lemma \ref{automorfismiquoziente} and Lemma \ref{automorfismiquoziente2}, the Lagrange theorem shows that there exists a non-negative integer $d$ such that 
$$|\aut({\mathbb{F}_{q^2}}(\cX))|=\ha(q^2-1)d.$$ 
From Result \ref{res10102024}, $\aut({\mathbb{F}_{q^2}}(\cX))$ is linear, that is, $\aut({\mathbb{F}_{q^2}}(\cX))$ is the subgroup of $PGL\left(5,\mathbb{F}_{q^2}\right)$ which leaves $\cX$ invariant. From the proof of Theorem \ref{thene}, $\mathcal{X}$ is contained in the intersection of two quadrics $\cQ_1,\,\cQ_2$ of equations (\ref{quadric1}) and (\ref{quadric2}), respectively.   

Since both $\cQ_1$ and $\cQ_2$ have degree 2, B\'ezout's theorem yields that their intersection $\cQ=\cQ_1\cap \cQ_2$ is a surface of degree 4. We show that $\cQ$ is irreducible. Otherwise, $\cQ$ splits into two components, say $\cQ=V_1\cup V_2$. Let $\pi$ be a generic hyperplane of $PG\left(4,\mathbb{F}\right)$ of equation $$\pi: \alpha_0+\alpha_1X_1+\alpha_2X_2+\alpha_3X_3+\alpha_4X_4=0.$$ Then $\pi\cap \cQ=\left(\pi\cap V_1\right) \cup \left(\pi\cap V_2\right)$ and $\pi\cap \cQ$ is a reducible variety for any choice of $\pi$, since $\pi\cap V_1$ and $\pi\cap V_2$ are varieties. Put $X_2=u$ and $X_3=v$. Then 
$$\left(1:\frac{u^2}{v}:u:v:v^2\right)$$ is the set of points in $\cQ$, and its intersection $\pi\cap Q$ with $\pi$ consists of all points of $\cQ$ satisfying equation  
$$\alpha_0 v+\alpha_1 u^2 +\alpha_2 uv+\alpha_3 v^2+\alpha_4 v^3=0.$$ Suppose that $\alpha_4=0$. In this case $\pi\cap \cQ$ can be viewed as a conic in the affine plane $AG(2,\mathbb{F})$ with coordinates $(u,v)$. Its associated matrix is
\begin{equation} \label{matriceassociata}
    \begin{pmatrix}
\alpha_1 & \ha\alpha_2 & 0 \\
\ha\alpha_2 & \alpha_3 & \ha\alpha_0\\
0 & \ha\alpha_0 & 0
    \end{pmatrix}.
\end{equation}
For a general choice of the coefficients of $\pi$, the matrix (\ref{matriceassociata}) has non-zero determinant, and hence the conic $\pi\cap \cQ$ is irreducible. Therefore, $\cQ$ itself is irreducible. 

Let $\alpha\in G=\aut({\mathbb{F}_{q^2}}(\mathcal{X}))$. We show that $\alpha$ is also an automorphism of $Q$. From $\alpha(\mathcal{X})=\mathcal{X}$, it follows that $\mathcal{X}\subseteq \alpha\left(\cQ\right)\cap \cQ$. Suppose that $\alpha(\cQ)\ne \cQ$. Then $\alpha(\cQ)\cap \cQ$ is a variety of dimension $1$ and degree at most $16$, since $\cQ$ and $\alpha(\cQ)$ has degree $4$. This is a contradiction, as $\mathcal{X}$ has degree $q+1>16$. Thus we have proved that $\alpha(\cQ)=\cQ$.

By direct computation it is shown that the singular points of $\cQ$ are $(1:0:0:0:0), (0:1:0:0:0)$, and $(0:0:0:0:1)$. 
Since $(0:1:0:0:0)\not\in\mathcal{X}$ and $G$ is linear, two possibilities may occur, namely either $G$ fixes both $(1:0:0:0:0)$ and $(0:0:0:0:1)$, or interchanges them. 

Suppose that $G$ has an element $\alpha$ of order $p$. Then $\alpha$ fixes both $(1:0:0:0:0)$ and $(0:0:0:0:1)$, since $p\ne 2$. By Results \ref{resth11.129} and \ref{zeroprank}, 
$\alpha$ fixes exactly one point, 
which gives a contradiction. Hence the order of $G$ is prime to $p$.

If $G$ fixes both $(1:0:0:0:0)$ and $(0:0:0:0:1)$,  \cite[Theorem 11.60]{HKT} yields $\ha(q^2-1)d\leq 4\mathfrak{g}+2,$ 
whence $q^2(d-1)+2q-5-d\leq 0$ which is impossible for $d>0$.
Hence $G$ has an orbit of length $2$. From the orbit-stabilizer theorem 
$|G|=2|G_O|$ with $O=(0:0:0:0:1)$.
Thus $|G_O|=\frac{1}{4}(q^2-1)d.$ 
By \cite[Theorem 11.60]{HKT},  $\frac{1}{4}(q^2-1)d\leq 4\mathfrak{g}+2$ whence $q^2(d-2)+4q-d-10\leq 0$. Since $q\ge 5$ this yields $d=1$, and hence 
$|G|=\ha(q^2-1)$. In other words, $G=\aut({\mathbb{F}_{q^2}}(\cX))$ is the inherited automorphism group.
\end{proof}
\begin{rem}
\label{rem16102024}\em{From the proof of Theorem \ref{teo10102024}, $Q=Q_1\cap Q_2$ is irreducible. However, $Q$ is not a Del Pezzo surface as $Q$ has singular points.}
\end{rem}
\section{On the geometry of the natural embedding} 
From Theorem \ref{lemWS1} and Result \ref{res25092024},  the order sequence of $\cX$ at $Q_\infty$ is $\left(0,1,\frac{1}{4}(q+3),\ha(q+1),q+1\right)$. By Proposition \ref{pro26092024}, this holds true for $Q_O$.

Next we determine the order sequence of $\cX$ at $Q_c$ with $c\neq 0$. From (\ref{ramoorgineY}), 
the first three rows in the Wronskian matrix are
\begin{equation*}
W(Q_c)=
\begin{pmatrix}
 1& 0 & 0 & \ast & \ast\\
 0& 0 & 1 & \ast & \ast\\
 0 & 2c & 0 & \ast & \ast\\
\end{pmatrix}
\end{equation*}
which shows that $W\left(Q_c\right)$ has rank at least three. Therefore, the first three orders are $(0,1,2)$. From Proposition \ref{pro26092024B}, $q+1-\ha(q-3)=\ha(q+5)$ is also an order. This together with Result \ref{res25092024} show that the order sequence of $\cX$ at $Q_c$ is $\left(0,1,2,\ha(q+5),q+1\right)$.

We go on by computing the Wronskian determinant of $\cX$ at other points. For simplicity, we carry out the computation for $p\ge 5$ as  $\cX$ is a curve embedded in a four dimensional projective space, and hence there is no need to use the Hasse derivates. Take a point $P\in \cX$  which is not the image of a branch of $\cF$ centered at the origin or at $P_\infty$. Then $P$ is the image of a unique branch of $\cF$. By a slight abuse of notation, we denote the center of that branch of $\cF$ by the same letter $P$.  
Therefore, $P=(a,b)\in\mathcal{C}$, $P\neq(0,0), P\neq P_\infty$.
A primitive branch representation of the unique branch of $\cF$ centered at $P$
is $(x(t),y(t))=\left(a+t,b+mt+b_2t^2+b_3t^3+\ldots\right)$. Therefore, 
$$x(0)=a,\,x^{(1)}(0)=1,\,x^{(i)}(0)=0,\quad i\ge 2.$$ 
Furthermore, 
$$ y(0)=b,\, y^{(1)}(0)=m,\, y^{(2)}(0)=2b_2,\, y^{(3)}(0)=6b_3. $$
Let  $$z(t)=\frac{y^2(t)}{x(t)}.$$ Then 
$$
z(0)=\frac{b^2}{a},\,\, z^{(1)}(0)=\frac{2amb-b^2}{a^2},\,\, z^{(2)}(0)=\frac{2(a^2m^2+2ab(ab_2-m)+b^2)}{a^3},
$$
$$
z^{(3)}(0)=\frac{(6b-4am+2b_2a^2)(2am-b)+(2ma^2-4ab)(4ab_2+m)+a^2b(12ab_3+6b_2)}{a^4}.
$$
Let 
$$u(t)=\left(\frac{y^2(t)}{x(t)}\right)^2.$$
Then
$$
    u(0)=\frac{b^4}{a^2}\,\,,
    u^{(1)}(0)=-\frac{2b^3(b-2am)}{a^3}\,\,,
    u^{(2)}(0)=\frac{2b^2(6a^2m^2+4ab(ab_2-2m)+3b^2)}{a^4}\,,$$
    $$
    u^{(3)}(0)= -\frac{24b^4}{a^5}+\frac{72b^3m}{a^4}-\frac{48b_2b^3+72b^2m^2}{a^3}+\frac{24b_3b^3+72b_2b^2m+24bm^3}{a^2}\,,$$
    
    $$u^{(4)}(0)=\frac{120b^4}{a^6}-\frac{384b^3m}{a^5}+\frac{288b_2b^3+432b^2m^2}{a^4}- \frac{192b_3b^3+576b_2b^2m+192bm^3}{a^3}+\frac{288b_3b^2m+144b_2^2b^2+288b_2bm^2+24m^4}{a^2}. $$
From the derivation of 
\begin{equation}
\label{eq24092024}
    y(t)^q+y(t)(a+t)^{\nicefrac{(q-1)}{2}}-(a+t)^{q-\nicefrac{(q-1)}{4}}=0
\end{equation}
yields
\begin{equation}
    \label{derivataprima}
    y^{(1)}(t)(a+t)^{\nicefrac{(q-1)}{2}}-\frac{1}{2}y(t)(a+t)^{\nicefrac{(q-3)}{2}}-\frac{1}{4}(a+t)^{\nicefrac{(3q-3)}{4}},
\end{equation}
\begin{equation}
    \label{derivataseconda}
    y^{(2)}(t)(a+t)^{\nicefrac{(q-1)}{2}}-y^{(1)}(t)(a+t)^{\nicefrac{(q-3)}{2}}+\frac{3}{4}y(t)(a+t)^{\nicefrac{(q-5)}{2}}+\frac{3}{16}(a+t)^{\nicefrac{(3q-7)}{4}}=0,
\end{equation}
\begin{equation}
    \label{derivataterza}
    y^{(3)}(t)(a+t)^{\nicefrac{(q-1)}{2}}-\frac{3}{2}y^{(2)}(t)(a+t)^{\nicefrac{(q-3)}{2}}+\frac{9}{4}y^{(1)}(t)(a+t)^{\nicefrac{(q-5)}{2}}-\frac{15}{8}y(t)(a+t)^{\nicefrac{(q-7)}{2}}-\frac{21}{64}(a+t)^{\nicefrac{(3q-11)}{4}}=0.
\end{equation}
For $t=0$ in (\ref{eq24092024}), (\ref{derivataprima}), (\ref{derivataseconda}), (\ref{derivataterza}) we have
\begin{equation}
    \label{sist1}
    \begin{cases}
         b^q+ba^{\nicefrac{(q-1)}{2}}-a^{q-\nicefrac{(q-1)}{4}}=0  \\
            ma^{\nicefrac{(q-1)}{2}}-\frac{1}{2}ba^{\nicefrac{(q-3)}{2}}-\frac{1}{4}a^{\nicefrac{(3q-3)}{4}}=0\\
             2b_2a^{\nicefrac{(q-1)}{2}}-ma^{\nicefrac{(q-3)}{2}}+\frac{3}{4}ba^{\nicefrac{(q-5)}{2}}+\frac{3}{16}a^{\nicefrac{(3q-7)}{4}}=0 \\
             6b_3 a^{\nicefrac{(q-1)}{2}}-3b_2 a^{\nicefrac{(q-3)}{2}}+\frac{9}{4}ma^{\nicefrac{(q-5)}{2}}-\frac{15}{8}ba^{\nicefrac{(q-7)}{2}}-\frac{21}{64}a^{\nicefrac{(3q-11)}{4}}=0.
    \end{cases}
    \end{equation}

From the second equation of (\ref{sist1}), 
\begin{equation}
\label{mespl}
m=\frac{1}{2}\frac{b}{a}+\frac{1}{4}a^{\nicefrac{(q-1)}{4}}.
\end{equation}
Substituting  $m$ from (\ref{mespl}) in the third equation of (\ref{sist1}) gives
\begin{equation}
    \label{b2espl}
    b_2=-\frac{1}{8}\frac{b}{a^2}+\frac{1}{32}a^{\nicefrac{(q-5)}{4}}.
\end{equation}
Thus (\ref{b2espl}) can also be written as
\begin{equation}
    \label{b2espll}
    \frac{1}{8}a(m)-\frac{3}{16}\frac{b}{a^2}.
\end{equation}
Substitution of (\ref{mespl}) and (\ref{b2espll}) in the fourth equation of (\ref{sist1}) gives
\begin{equation}
    \label{b3espl}
    b_3=\frac{1}{16}\frac{b}{a^3}-\frac{3}{128}a^{\nicefrac{(q-9)}{4}}.
\end{equation}
For $b_2=0$, (\ref{b2espl}) yields  
\begin{equation}
    \label{besplicit}
    b=\frac{1}{4}a^{\nicefrac{(q+3)}{4}}.
\end{equation}
Substitution of (\ref{besplicit}) in the first equation of (\ref{sist1}) gives $a^{\nicefrac{(q^2-1)}{4}}=3$. Equation $x^{\nicefrac{(q^2-1)}{4}}=81$ has a finite number of solutions, and we assume $a$ not to be any of them.


For $b_3=0$, (\ref{b3espl}) yields 
\begin{equation}\label{besplici}
    b=\frac{3}{8}a^{\nicefrac{(q+3)}{4}}.
\end{equation} 
Substitution of (\ref{besplici}) in the first equation of (\ref{sist1}) gives $a^{\nicefrac{(q^2-1)}{4}}=\frac{5}{3}$. Equation $x^{\nicefrac{(q^2-1)}{4}}=\frac{5}{3}$ has a finite number of solutions, and we assume $a$ not to be any of them.  


Therefore, apart from the above finitely many discarded values of $a$, 
the Wronskian matrix of $\cX$ at the point $P$ is 
$$W=\begin{pmatrix}    1 & a & b & \frac{b^2}{a} & \frac{b^4}{a^2} \\
    0 & 1 & m & \ldots & \ldots \\
    0 & 0 & 2b_2 & \ldots & \ldots \\
    0 & 0 & 6b_3 & \ldots & \ldots \\
    0 & 0 & \ldots & \ldots & \ldots 
\end{pmatrix}.$$
Since 
$$
\begin{vmatrix}
    1& a\\
    0& 1
\end{vmatrix}=1,\qquad 
\begin{vmatrix}
    1& a & b\\
    0& 1 & m\\
    0 & 0 & 2b_2
\end{vmatrix}=2b_2\neq 0,$$
$W$ has rank at least $3$. We show that $W$ has rank $4$. For this purpose, 
compute the determinant whose rows and columns are the first four. This determinant $D$ is equal to $2b_2z^{(3)}(0)-6b_3z^{(2)}(0)$, and hence 
\begin{equation*} 
D=
\begin{vmatrix}
 1& a & b & \frac{b^2}{a}\\
    0& 1 & m & \frac{(2amb-b)}{a^2}\\
    0 & 0 & 2b_2 & \frac{2(a^2m^2+2ab(ab_2-m)+b^2)}{a^3}\\
    0 & 0 & 6b_3 & \frac{(6b-4am+2b_2a^2)(2am-b)+(2ma^2-4ab)(4ab_2+m)+a^2b(12ab_3+6b_2)}{a^4}
\end{vmatrix}
\end{equation*}
We verify that $D\ne 0$. 

From (\ref{mespl}), (\ref{b2espl}) and (\ref{b3espl}) the derivative and the higher derivatives of $z$ evaluated at $P$, i.e. in $t=0$, are 
\begin{equation*}
   z^{(1)}(0)=\frac{1}{2}ba^{\nicefrac{(q-5)}{4}}
\end{equation*}
\begin{equation}
   \label{zsecab} 
    z^{(2)}(0)=-\frac{3}{8}ba^{\nicefrac{(q-9)}{4}}+\frac{1}{8}a^{\nicefrac{(q-3)}{2}}
   \end{equation}
\begin{equation}
    \label{zterab}
    z^{(3)}(0)=-\frac{9}{32}a^{\nicefrac{(q-5)}{2}}+\frac{21}{32}ba^{\nicefrac{(q-13)}{4}}.
\end{equation}
Similarly, from (\ref{mespl}), (\ref{b2espl}) and (\ref{b3espl}), one can compute  the derivative and higher derivatives of $u$ evaluated at $P$. 
\begin{equation*}
   u^{(1)}(0)=b^2a^{\nicefrac{(q-9)}{4}}
\end{equation*}
\begin{equation}
   \label{usecab} 
    u^{(2)}(0)=-\frac{3}{4}b^3a^{\nicefrac{(q-13)}{4}}+\frac{3}{4}b^2a^{\nicefrac{(q-5)}{2}}
   \end{equation}
\begin{equation}
    \label{uterab}
    u^{(3)}(0)=-\frac{27}{16}b^2a^{\nicefrac{(q-7)}{2}}+\frac{3}{8}ba^{\nicefrac{(3q-11)}{4}}+\frac{21}{16}b^3a^{\nicefrac{(q-17)}{4}}
\end{equation}

Now, by Equations (\ref{zsecab}) and (\ref{zterab}), 
\begin{equation} D=
    \label{detab}
    -\frac{3}{128}b^2a^{\nicefrac{(q-21)}{4}}+\frac{3}{256}ba^{\nicefrac{(q-9)}{2}}.
\end{equation}
Now, if $D=0$, then 
$$b=\frac{1}{2}a^{\nicefrac{(q+3)}{4}}$$ as 
$b$ can be computed from $a$ by (\ref{detab}), since $a,b\neq 0$ and $p\neq 3$. 
This together with $b^q+ba^{\nicefrac{(q-1)}{2}}=a^{q-\nicefrac{(q-1)}{4}}$ yields
$a^{\nicefrac{(q^2-1)}{4}}=1$. Equation $x^{\nicefrac{(q^2-1)}{4}}=1$ has a finite number of solutions, and we assume $a$ not to be any of them. In turns out that apart from the above discarded values of $a$, determinant $D$ does not vanish. Thus for all but a finitely many points of $\cX$, the rank of $W$ is at least four.   
\begin{theorem}
For $p\ge 5$, the order sequence of $\cX$ is 
$\left(\varepsilon_0,\varepsilon_1,\varepsilon_2,\varepsilon_3,\varepsilon_4\right)=\left(0,1,2,3,q\right)$. 
\end{theorem}
\begin{proof} 
From the above computation, for all but a finitely many points $W$ has rank at least four, that is, $\left(\varepsilon_0,\varepsilon_1,\varepsilon_2,\varepsilon_3\right)=\left(0,1,2,3\right)$.
Furthermore, $\varepsilon_4=q$ by Result \ref{res25092024}. 
\end{proof}
Our method can also be used to determine the order sequence at any point we have discarded including the Weierstrass points of $\cX$. 
Unfortunately, some details require long calculations and we limit ourselves to compute determinant $E=2b_2u^{(3)}(0)-6b_3u^{(2)}(0)$ of the first four rows and the first, second, third and fifth columns: 
\begin{equation*}
E=
\begin{vmatrix}
 1& a & b & \frac{b^4}{a^2}\\
    0& 1 & m & -\frac{2b^3(b-2am)}{a^3}\\
    0 & 0 & 2b_2 & \frac{2b^2(6a^2m^2+4ab(ab_2-2m)+3b^2)}{a^4}\\
    0 & 0 & 6b_3 & -\frac{24b^4}{a^5}+\frac{72b^3m}{a^4}-\frac{48b_2b^3+72b^2m^2}{a^3}+\frac{24b_3b^3+72b_2b^2m+24bm^3}{a^2}
\end{vmatrix}
\end{equation*}
By (\ref{usecab}) and (\ref{uterab}), 
\begin{equation} E=
    \label{detabu}
     \frac{3}{128}ba^{\nicefrac{(q-25)}{4}}\left(-2b^3+5b^2a^{\nicefrac{(q+3)}{4}}-4ba^{\nicefrac{(q+3)}{2}}+a^{\nicefrac{(3q+9)}{4}}\right).
\end{equation}
If $E=0$ then (\ref{detabu}) yields 
\begin{equation*}
    -2b^3+5b^2a^{\nicefrac{(q+3)}{4}}-4ba^{\nicefrac{(q+3)}{2}}+a^{\nicefrac{(3q+9)}{4}}= \left(-2b+a^{\nicefrac{(q-3)}{4}}\right)\left(a^{\nicefrac{(q-3)}{4}}-b\right)^2.
\end{equation*}
Since $a,b\neq 0$ and $p\neq 3$, this implies either 
\begin{equation} \label{cond1}
    b=a^{\nicefrac{(q+3)}{4}}
\end{equation}
or
\begin{equation*}
    b=\frac{1}{2}a^{\nicefrac{(q+3)}{4}}.
\end{equation*}
The former case cannot actually occur as (\ref{cond1}) together with (\ref{quozientebella}) imply $a=0$.

From the above discussion, the Weierstrass points of $\cX$ are among the Frobenius images of the (non-singular) points $P=(1:a:b)$ of $\cF$ where
$$(a,b)=\left\{a^{(q^2-1)/4}=3,b=\frac{1}{4}a^{\nicefrac{(q+3)}{4}}\right\}\cup\left\{a^{(q^2-1)/4}=\frac{5}{3},b=\frac{3}{8}a^{\nicefrac{(q+3)}{4}}\right\}\cup\left\{a^{(q^2-1)/4}=1, b=\frac{1}{2}a^{\nicefrac{(q+3)}{4}}\right\}    
$$
together with the points of $\cX$ arising from the singular points of $\cF$; see Proposition \ref{propoints}. 
\begin{theorem}
For $p\ge 7$, the $\mathbb{F}_{q^2}$-rational Weierstrass points of $\cX$ are those points which arise from the singular points of $\cF$.  
\end{theorem}
\begin{proof} It is enough to notice that if $p\ge 7$ then $a^{\nicefrac{(q^2-1)}{4}}$ with $a\in \mathbb{F}_{q^2}$ is always distinct from $1,3$ and $\frac{5}{3}.$  
\end{proof}

\end{document}